\documentclass[final,leqno]{siamltex704}
\usepackage{amsmath}
\usepackage{graphicx}
\usepackage[notcite,notref]{showkeys}
\usepackage{mathrsfs}
\usepackage{float}
\usepackage{amsfonts,amssymb}
\usepackage{dsfont}
\usepackage{pifont}
\usepackage{wrapfig} 
\usepackage{hyperref}
\usepackage{multirow}
\numberwithin{equation}{section}
\def\3bar{{|\hspace{-.02in}|\hspace{-.02in}|}}
\def\E{{\mathcal{E}}}
\def\T{{\mathcal{T}}}

\def\O{{\mathcal{O}}}

\def\pT{{\partial T}}

\def\W{{\mathcal{W}}}

\def\bw{{\mathbf{w}}}

\def\bn{{\mathbf{n}}}

\def\bx{{\mathbf{x}}}

\def\ljump{{[\![}}
\def\rjump{{]\!]}}

\newtheorem{algorithm}{Primal-Dual Weak Galerkin Algorithm}[section]

\title{A New Primal-Dual Weak Galerkin Finite Element Method for Ill-posed Elliptic Cauchy Problems}

\author{Chunmei Wang \thanks{Department of Mathematics \& Statistics, Texas Tech University, Lubbock, TX 79409, USA. The research of Chunmei Wang was partially supported by National Science Foundation Award DMS-1849483 and DMS-1648171.}
}

\begin{document}

\maketitle

\begin{abstract}
A new numerical method is devised and analyzed for a type of ill-posed elliptic Cauchy problems by using the primal-dual weak Galerkin finite element method. This new primal-dual weak Galerkin algorithm is robust and efficient in the sense that the system arising from the scheme is symmetric, well-posed, and is satisfied by the exact solution (if it exists). An error estimate of optimal order is established for the corresponding numerical solutions in a scaled residual norm. In addition, a mathematical convergence is established in a weak $L^2$ topology for the new numerical method. Numerical results are reported to demonstrate the efficiency of the primal-dual weak Galerkin method as well as the accuracy of the numerical approximations.
\end{abstract}



\pagestyle{myheadings}

\section{Introduction}
This paper is concerned with the development of new numerical methods for solving a type of elliptic Cauchy problems. For simplicity, we consider the second order elliptic equation with Cauchy boundary data on part of the boundary: Find an unknown function $u=u(\bx)$ such that
\begin{equation}\label{1}
\left\{
\begin{split}
-\nabla(a\nabla u) =&f\qquad\ \text{in}\ \Omega,\\
u =& g_1\qquad \text{on}\  \Gamma_d,\\
a\nabla u \cdot \bn=& g_2\qquad \text{on}\ \Gamma_n,
\end{split}\right.
\end{equation}
where $\Omega$ is an open bounded domain in $\mathbb R^d$ $(d=2, 3)$ with Lipschitz continuous boundary $\partial\Omega$; $\Gamma_d$ and $\Gamma_n$ are two segments of the domain boundary $\partial\Omega$; $f\in L^2(\Omega)$; the Cauchy data $g_1\in H^{\frac12}(\Gamma_d)$ and $g_2\in (H_{00}^{1/2}(\Gamma_n))'$ are two given functions defined on $\Gamma_d$ and $\Gamma_n$, respectively; $\bn$ is an unit outward normal direction to $\Gamma_n$. The diffusion coefficient $a=a(\bx)$ is assumed to be symmetric, bounded, and uniformly positive definite in the domain $\Omega$.

The essence of the elliptic Cauchy problem is to solve a partial differential equation on a domain with over-specified boundary conditions given on parts of the domain boundary. On the other side, the elliptic Cauchy problem is to solve a data completion problem with missing boundary conditions on the remaining parts of the domain boundary. It is well-known that the solution of the elliptic Cauchy problem \eqref{1} (if it exists) must be unique, provided that $\Gamma_d\cap\Gamma_n$ is a nontrivial portion of $\partial \Omega$. Throughout this paper, we assume that the Cauchy data is compatible so that the solution of the elliptic Cauchy problem (\ref{1}) exists; furthermore, we assume that $\Gamma_d\cap\Gamma_n$ is a nontrivial portion of $\partial \Omega$ so that the solution of the elliptic Cauchy problem (\ref{1}) is unique.

The elliptic Cauchy problems arise from various applications in science and engineering, such as vibration, wave propagation, geophysics, electromagnetic scattering,  steady-state inverse heat conduction, cardiology and nondestructive testing; etc. Readers are referred to the ``Introduction Section'' in \cite{ww2018} and the references cited therein for a detailed description of the elliptic Cauchy problems.

This paper aims to devise a new numerical scheme for the elliptic Cauchy problem \eqref{1} by using the newly developed primal-dual weak Galerkin (PD-WG) finite element method \cite{ww2016, ww2017, ww2018}. The new scheme is different from the one introduced and analyzed in \cite{ww2018} although both aim to solve numerically the elliptic Cauchy problem under the general framework of the PD-WG finite element method. The main difference between them is that the present approach is based on weak gradients while the scheme in \cite{ww2018} is based on weak Laplacians. As a result, the weak finite element space consisting of piecewise linear functions is applicable in the present approach, but not in \cite{ww2018}, as the Laplacian of linear functions would be vanishing. In addition, new mathematical tools (namely, methods based on generalized inf-sup conditions) must be introduced in order to establish a mathematical theory for the new scheme.

Let us now briefly introduce the essential ideas behind the PD-WG finite element method for solving the elliptic Cauchy problem \eqref{1}. Denote by $\Gamma_n^c= \partial\Omega \setminus \Gamma_n$, and by $H^1_{0,\Gamma_n^c}(\Omega)$ the subspace of $H^1(\Omega)$ consisting of functions with homogeneous boundary value on $\Gamma_n^c$; i.e.,
$$
H^1_{0,\Gamma_n^c}(\Omega)=\{v\in H^1(\Omega): v=0 \ \text{on}\  \Gamma_n^c \}.
$$
A weak formulation for the elliptic Cauchy problem \eqref{1} would find $u\in H^{1}(\Omega)$ such that $u= g_1$ on $\Gamma_d$ and
\begin{equation}\label{weakform-01}
(a\nabla u,\nabla w)=\langle g_2, w\rangle_{\Gamma_n}+(f, w) \qquad \forall w\in H^1_{0,\Gamma_n^c}, \end{equation}
where $\langle\cdot, \cdot\rangle_{\Gamma_n}$ stands for the pairing between $H^{\frac12}_{00}(\Gamma_n)$ and $(H^{\frac12}_{00}(\Gamma_n))'$. The weak formulation (\ref{weakform-01}) is different from the one employed in \cite{ww2018}, and shall result in a new numerical scheme different from the one in \cite{ww2018} although both use the general framework of the primal-dual approach.

Using the weak gradient operator $\nabla_w$ introduced originally in \cite{wy3655}, one may reformulate \eqref{weakform-01} as follows
\begin{equation}\label{weakform-02}
(a\nabla_w \{u\},\nabla_w \{w\})=\langle g_2, w\rangle_{\Gamma_n}+(f, w) \qquad \forall w\in H^1_{0,\Gamma_n^c},
\end{equation}
where $\{u\} = \{u|_T, u|_\pT\}$ and $\{w\} = \{w|_T, w|_\pT\}$  are the weak functions (see Section \ref{Section:WGFEM} for the definition). The weak functions are then approximated by piecewise polynomials on each element $T$ and its boundary $\pT$. Note that no continuity requirement is necessary between the information in the element $T$ and on its boundary $\pT$. The weak gradient operator $\nabla_w$ is further discretized by using vector-valued polynomials, denoted as $\nabla_{w, h}$ (see $\nabla_{w, r, T}$ in Section \ref{Section:WGFEM} for its precise definition) so that the weak form (\ref{weakform-02}) can be approximated by
\begin{equation}\label{EQ:10-12-2015:01}
(a\nabla_{w, h} u,\nabla_{w, h} w)=\langle g_2, w\rangle_{\Gamma_n}+(f,w_0) \qquad \forall w\in V^{h}_{0, \Gamma_n^c},
\end{equation}
where $V^{h}_{0, \Gamma_n^c}$ is a test space consisting of weak finite element functions with proper boundary values. However, the discrete problem (\ref{EQ:10-12-2015:01}) is not well-posed unless the {\em inf-sup} condition of Babu\u{s}ka \cite{babuska} and Brezzi \cite{b1974} is satisfied. The primal-dual formulation is thus developed to overcome this difficulty through a strategy that couples \eqref{EQ:10-12-2015:01} with its dual equation which seeks $\lambda_h\in V^{h}_{0, \Gamma_n^c}$ satisfying
\begin{equation}\label{EQ:09-12-2018:01}
(a\nabla_{w, h} v, \nabla_{w, h} \lambda_h) =0\qquad
\forall v \in V^{h}_{0, \Gamma_d}.
\end{equation}
A formal coupling between \eqref{EQ:10-12-2015:01} and \eqref{EQ:09-12-2018:01} can be accomplished via a stabilizer, denoted as $s(v,v)$, designed to measure the level of ``continuity" of $v\in V^{h}$ in the sense that $v\in V^{h}$ is
a classical $C^0$-conforming element if and only if $s(v,v)=0$.
The resulting scheme seeks $u_h\in V^{h}$ and $\lambda_h\in V^{h}_{0, \Gamma_n^c}$ satisfying $u_b=Q_b g_1$ on $\Gamma_d$, and the following equations:
\begin{equation}\label{primal-dual-wg}
\left\{
\begin{split}
s(u_h, v) - (a\nabla_{w, h} v, \nabla_{w, h} \lambda_h) &=0\qquad\qquad
\forall v \in V^{h}_{0, \Gamma_d},\\
s(\lambda_h, w) + (a\nabla_{w,h} u_h, \nabla_{w,h} w) &= (f, w_0)+\langle g_2, w_b\rangle_{\Gamma_n}\qquad\forall
 w \in V^{h}_{0, \Gamma_n^c},
 \end{split}\right.
\end{equation}
where $s(\cdot,\cdot)$ is a bilinear form in the weak finite element space $V^{h}$ known as the {\em stabilizer} or {\em smoother} that enforces certain weak continuity for the approximation $u_h$ and $\lambda_h$. Numerical schemes in the form of (\ref{primal-dual-wg}) have been named {\em primal-dual weak Galerkin finite element methods} in \cite{ww2016, ww2017, ww2018}, and they are also known as {\em stabilized finite element methods} in \cite{Burman, ErikBurman-EllipticCauchy-SIAM02, ErikBurman-EllipticCauchy} in different finite element contexts.

The primal-dual weak Galerkin finite element method \eqref{primal-dual-wg} has shown promising features as a discretization approach in the following aspects: (1) it offers a symmetric and well-posed problem for the ill-posed elliptic Cauchy problem; (2) it is consistent in the sense that the exact solution (if it exists) satisfies the system; (3) it works well for a wide class of PDE problems for which no traditional variational formulations are available; and (4) it admits general finite element partitions consisting of arbitrary polygons or polyhedra.

The paper is organized as follows. In Section \ref{Section:WGFEM}, we introduce a primal-dual weak Galerkin finite element scheme for solving the elliptic Cauchy problem (\ref{1}). In Section \ref{Section:Stability}, we present some technical estimates and a generalized inf-sup condition useful for a mathematical study of the new algorithm. In Section \ref{Section:EE}, we derive an error equation for the numerical solutions obtained from the primal-dual weak Galerkin algorithm devised in Section \ref{Section:WGFEM}. In Section \ref{Section:ErrorEstimate}, we establish an optimal order error estimate for the primal-dual WG finite element approximations in a scaled residual norm. In Section \ref{Section:L2Error}, a convergence theory in a weak $L^2$ topology is presented under a certain regularity assumption for the elliptic Cauchy problem. Finally in Section \ref{Section:NE}, we report some numerical results to demonstrate the efficiency and accuracy of our new PD-WG finite element method.

We follow the usual notations for Sobolev spaces and norms. For any open bounded domain $T\subset \mathbb{R}^d$ ($d$-dimensional Euclidean space) with Lipschitz continuous boundary, we use $\|\cdot\|_{s,T}$ and $|\cdot|_{s,T}$ to denote the norm and seminorm in Sobolev space $H^s(T)$ for any $s\ge 0$, respectively. The inner product in $H^s(T)$ is denoted by $(\cdot,\cdot)_{s,T}$. The space $H^0(T)$ coincides with $L^2(T)$, for which the norm and the inner product are denoted by $\|\cdot \|_{T}$ and $(\cdot,\cdot)_{T}$, respectively. For the case that $T=\Omega$, we shall drop the subscript $T$ in the norm and inner product notations. Throughout the paper, $C$ appearing in different places stands for different constants.

\section{Primal-Dual Weak Galerkin}\label{Section:WGFEM}
Denote by ${\cal T}_h$ a finite element partition of the domain $\Omega \subset \mathbb R^d (d=2, 3)$ into polygons in 2D or polyhedra in 3D which is shape regular in the sense described in \cite{wy3655}. A weak function on the element $T\in {\cal T}_h$ refers to a pair $v=\{v_0,v_b \}$ where $v_0\in L^2(T)$  and $v_b\in L^{2}(\partial T)$. The component $v_0$ can be understood as the value of $v$ in $T$, and the other component $v_b$ represents the value of $v$ on the boundary $\partial T$. Note that $v_b$ may not necessarily be the trace of $v_0$  on $\partial T$, though $v_b=v_0|_{\partial T}$ would be a feasible option. Denote by $\W(T)$ the space of all weak functions on $T$; i.e.,
\begin{equation*}\label{2.1}
\W(T):=\{v=\{v_0,v_b \}: v_0\in L^2(T), v_b\in
L^{2}(\partial T) \}.
\end{equation*}

Denote by $P_r(T)$ the set of polynomials on $T$ with degree no more than $r\ge 0$. A discrete weak gradient of $v\in \W(T)$, denoted as $\nabla_{w,r,T}v$, is defined as the unique polynomial vector in $[P_r(T) ]^d$ satisfying
\begin{equation}\label{disgradient}
(\nabla_{w,r,T}  v, \boldsymbol{\psi})_T=-(v_0,\nabla \cdot \boldsymbol{\psi})_T+\langle v_b, \boldsymbol{\psi} \cdot \textbf{n}\rangle_{\partial T} \quad\forall\boldsymbol{\psi}\in [P_r(T)]^d,
\end{equation}
which, from the usual integration by parts, gives
\begin{equation}\label{disgradient*}
(\nabla_{w,r,T} v, \boldsymbol{\psi})_T= (\nabla v_0, \boldsymbol{\psi})_T-\langle v_0- v_b, \boldsymbol{\psi} \cdot \textbf{n}\rangle_{\partial T} \quad\forall\boldsymbol{\psi}\in [P_r(T)]^d
\end{equation}
provided that $v_0\in H^1(T)$. The concept of discrete weak gradient was introduced originally in \cite{WangYe_2013,wy3655}.

Denote by ${\mathcal E}_h$ the set of all edges or flat faces in ${\cal T}_h$ and  ${\mathcal E}_h^0={\mathcal E}_h \setminus \partial\Omega$ the set of all interior edges or flat faces. Denote by $h_T$ the meshsize of $T\in {\cal T}_h$ and $h=\max_{T\in {\cal T}_h}h_T$ the meshsize for the partition ${\cal T}_h$.

For any given integer $k\geq 1$, denote by $V_k(T)$ the local discrete weak function space given by
$$
V_k(T)=\{\{v_0,v_b \}:v_0\in P_k(T),v_b\in
P_{k}(e), e\subset \partial T\}.
$$
Patching $V_k(T)$ over all the elements $T\in {\cal T}_h$ through a common value $v_b$ on the interior interface $\E_h^0$ gives rise to a weak finite element space $V^h$; i.e.,
$$
V^h=\big\{\{v_0,v_b \}:\{v_0,v_b \}|_T\in
V_k(T), \forall T\in {\cal T}_h
 \big\}.
$$
Denote by $V^h_{0,\Gamma_d}$ and $V^h_{0,\Gamma_n^c}$ the subspaces of $V^h$ with vanishing boundary value for $v_b$ on $\Gamma_d$ and $\Gamma_n^c$,  respectively; i.e.,
$$
V^h_{0,\Gamma_d}=\big\{\{v_0,v_b \}\in V^h:  v_b|_e=0,   e\subset \Gamma_d
 \big\},
$$
$$
V^h_{0,\Gamma_n^c}=\big\{\{v_0,v_b \}\in V^h:  v_b|_e=0,   e\subset \Gamma_n^c
 \big\}.
$$

\subsection{Algorithm}
For simplicity of notation and without confusion,  for any $\sigma\in V_h$, denote by $\nabla_{w}\sigma$  the discrete weak gradient $\nabla _{w,k-1,T}\sigma$ computed  by using (\ref{disgradient})  on each element $T$; i.e.,
$$
(\nabla _{w} \sigma)|_T=\nabla _{w,k-1,T}(\sigma|_T), \qquad
\sigma\in V^h.
$$

For any $u$, $v\in V^h$, we introduce the following bilinear forms
\begin{align*}
s(u, v)=&\sum_{T\in {\cal T}_h}h_T^{-1}
\langle  u_0-u_b,  v_0-v_b\rangle_{\partial T},
\\
b(u,v)=&\sum_{T\in {\cal T}_h}( a\nabla_w u, \nabla_w v)_T.
\end{align*}

Let $k\geq 1$ be an integer. For each element $T\in {\cal T}_h$, denote by $Q_0$ the $L^2$ projection onto $P_k(T)$. Denote by $Q_b$ the $L^2$ projection onto $P_{k}(e) $ for each edge or flat face $e\in {\cal E}_h$. For any $w\in H^1(\Omega)$, denote by $Q_h w$ the $L^2$  projection onto the weak finite element space $V^h$ such that on each element $T$,
$$
Q_hw=\{Q_0w,Q_bw\}.
$$
Denote by ${\cal Q}_h$ the $L^2$ projection onto the space of piecewise polynomials of degree $k-1$.

The numerical scheme for the elliptic Cauchy model problem (\ref{1})  based on the variational formulation (\ref{weakform-01}) by using primal-dual  weak Galerkin strategy is as follows:
\begin{algorithm}
Find $(u_h;\lambda_h)\in V^h
\times V_{0,\Gamma_n^c}^h$ satisfying $u_b= Q_b g_1$ on $\Gamma_d$,  such that
\begin{eqnarray}\label{32}
s( u_h, v)-b(v, \lambda_h)&=&0,\qquad\qquad\qquad\quad\qquad \forall v \in V_{0,\Gamma_d}^h,\\
s ( \lambda_h, w)+b(u_h, w)&=&(f,w_0)+\langle g_2,w_b\rangle_{\Gamma_n}, \quad \forall w\in V_{0,\Gamma_n^c}^h. \label{2}
\end{eqnarray}
\end{algorithm}

\subsection{Solvability}
The following is a well-known result on the solution's uniqueness for elliptic Cauchy problems, see \cite{Gilbarg-Trudinger} for reference.

\begin{lemma}\label{unilem}  Assume that $\Omega$ is an open bounded and connected domain in $\mathbb R^d \ (d=2,3)$ with Lipschitz continuous boundary $\partial\Omega$. Denote by $\Gamma_d$ the portion of the Dirichlet boundary and $\Gamma_n$ the Neumann portion. Assume that $\Gamma_d\cap \Gamma_n$ is a non-trivial portion of $\partial\Omega$. Then, the solutions of the following elliptic Cauchy problem, if they exist, are unique
\begin{equation*}
\begin{split}
-\nabla\cdot(a\nabla u)=&f, \qquad \ \text{in} \quad \Omega,\\
  u=&g_1,\qquad \text{on} \quad   \Gamma_d,\\
a\nabla u \cdot \bn=&g_2, \qquad \text{on} \quad  \Gamma_n.
\end{split}
\end{equation*}
\end{lemma}

\begin{lemma}\label{Lemma5.1} \cite{WangYe_2013,wy3655} The $L^2$ projection operators $Q_h$ and ${\cal Q}_h$ satisfy the following commutative property:
\begin{equation}\label{l}
\nabla _{w}(Q_h u) = {\cal Q}_h(\nabla u),\qquad u\in H^1(T).
\end{equation}
\end{lemma}

\begin{theorem}\label{thmunique1} Assume that $\Gamma_d\cap\Gamma_n$ contains a nontrivial portion of the domain boundary $\partial\Omega$ and $\Gamma_d\cap \Gamma_n\Subset \partial\Omega$ is a proper closed subset.  The primal-dual weak Galerkin algorithm (\ref{32})-(\ref{2}) has a unique solution.
\end{theorem}
\begin{proof} As the number of equations is the same as the number of unknowns, it suffices to show that the homogeneous problem  (\ref{32})-(\ref{2}) has only the trivial solution. To this end, we assume $f= 0$, $g_1=0$ and  $g_2=0$ in (\ref{32})-(\ref{2}). By letting $v=u_h$ and $w=\lambda_h$, the sum of (\ref{2}) and (\ref{32}) gives
$$
s(u_h, u_h)+ s(\lambda_h, \lambda_h)=0,
$$
which implies $u_0=u_b$ and  $\lambda_0= \lambda_b$ on each $\partial T$. Thus, we arrive at $u_0\in C^0(\Omega)$ and $ \lambda_0\in C^0(\Omega)$. Note that $u_b=0$ on $\Gamma_d$ and $ \lambda_b=0$ on $\Gamma_n^c$. Thus, (\ref{2}) can be rewritten as follows
$$
\sum_{T\in {\cal T}_h}(a\nabla_w u_h, \nabla_w w)_T=0, \qquad \forall w\in  V_{0,\Gamma_n^c}^h,
$$
which, by letting  $\boldsymbol{\psi}= \nabla_w u_h$ in  (\ref{disgradient}), gives rise to
\begin{equation}\label{uni}
\begin{split}
0=&\sum_{T\in {\cal T}_h} (a\nabla_w u_h, \nabla_w w)_T \\
=&\sum_{T\in {\cal T}_h} - (\nabla \cdot (a\nabla_w u_h), w_0)_T+\langle  a\nabla_w u_h \cdot \bn, w_b\rangle_{\partial T}\\
=&\sum_{T\in {\cal T}_h} - (\nabla \cdot (a\nabla_w u_h), w_0)_T+\sum_{e\in \E_h\setminus\Gamma_n^c}\langle  \ljump a\nabla_w u_h \cdot \bn\rjump, w_b\rangle_{e},
\end{split}
\end{equation}
where we have used $w_b=0$ on $\Gamma_n^c$. By letting $w_b=0$ on each edge $e\in \E_h \setminus \Gamma_n^c$ in (\ref{uni}) and $w_0=\nabla \cdot (\nabla_w u_h)$, we arrive at $\nabla \cdot (a\nabla_w u_h)=0$ on each element $T\in {\cal T}_h$. Similarly, by letting $w_0=0$ on each element $T\in {\cal T}_h$ in (\ref{uni}), we obtain $ \ljump \nabla_w u_h \cdot \bn\rjump=0$ on each edge or face $e\in \E_h\setminus\Gamma_n^c$.

Note that $u_0=u_b$ on each $\partial T$.   It follows
from (\ref{disgradient*}) that  for any $\textbf{q}\in  [P _{k-1}(T) ]^d$, we have \begin{equation*}
\begin{split}
(\nabla_w u_h, \textbf{q})_T&=(\nabla u_0, \textbf{q})_T+\langle \textbf{q} \cdot \bn, u_b-u_0\rangle_{\partial T} \\
 &=(\nabla u_0, \textbf{q})_T,
\end{split}
\end{equation*}
which gives rise to  $\nabla_w u_h=\nabla u_0$ on each element $T\in {\cal T}_h$. This implies that $\nabla\cdot (a\nabla u_0)=\nabla\cdot (a\nabla_w u_h) =0$ on each element $T\in {\cal T}_h$. Using $ \ljump \nabla_w u_h \cdot \bn\rjump=0$ on each edge or face $e\in \E_h \setminus \Gamma_n^c$, we arrive at  $ \ljump \nabla  u_0 \cdot \bn\rjump=0$  on each $e\in \E_h\setminus\Gamma_n^c$. Note that we have $u_0=0$ on $\Gamma_d$ and $\nabla u_0\cdot \bn=0$ on $\Gamma_n$, and $\Gamma_d\cap \Gamma_n$ is non-trivial portion of $\partial\Omega$. Thus, from Lemma \ref{unilem} we obtain $u_0\equiv 0$  in $\Omega$. Using $ u_0=u_b$ on each $\partial T$ gives $u_b=0$ on each  $\partial T$. Thus, $u_h\equiv 0$ in $\Omega$.

Since $\Gamma_d \cup \Gamma_n \Subset\partial \Omega$ is a closed proper subset, then $\Gamma_d^c \cap \Gamma_n^c= (\Gamma_d \cup \Gamma_n)^c$ contains a nontrivial portion of $\partial \Omega$. A similar argument can be made to show that $\lambda_h\equiv 0$ in $\Omega$. This completes the proof of the theorem.
\end{proof}

\section{Some Technical Estimates}\label{Section:Stability}
The goal of this section is to establish some technical results which are valuable in the error analysis of the primal-dual weak Galerkin finite element method (\ref{32})-(\ref{2}) for solving the elliptic Cauchy problem (\ref{1}).

In the weak finite element space $V^h$, we introduce four semi-norms as follows:
\begin{equation}\label{norm1}
\3bar v  \3bar_{ \Gamma_d}=\Big( \sum_{T\in {\cal T}_h} h_T^2 \| \nabla\cdot (a\nabla v_0)\|_T^2 + \sum_{e\in \E_h\setminus{\Gamma_n^c}}  h_T \|\ljump a\nabla v_0\cdot\bn\rjump\|_e^2+ s(v, v)\Big)^{\frac{1}{2}},
\end{equation}

\begin{equation}\label{norm}
\3bar v \3bar _{ h, \Gamma_d }=\Big(\sum_{T\in {\cal T}_h}  h_T^2\| \nabla\cdot (a\nabla_w v )\|_T^2 + \sum_{e\in \E_h\setminus \Gamma_n^c}  h_T\|\ljump a\nabla_w v \cdot\bn\rjump\|_e^2+ s(v , v)\Big)^{\frac{1}{2}},
\end{equation}

\begin{equation}\label{normgamma1}
\3bar \lambda \3bar_{ \Gamma_n^c }=\Big( \sum_{T\in {\cal T}_h} h_T^2  \| \nabla\cdot (a\nabla \lambda_0)\|_T^2 +\sum_{e\in\E_h\setminus{\Gamma_d}} h_T \|\ljump a\nabla \lambda_0\cdot\bn\rjump\|_e^2+ s(\lambda , \lambda)\Big)^{\frac{1}{2}},
\end{equation}

\begin{equation}\label{normgamma11}
\3bar \lambda  \3bar_{h, \Gamma_n^c}=\Big(\sum_{T\in {\cal T}_h}  h_T^2\| \nabla\cdot (a\nabla_w \lambda )\|_T^2 +
 \sum_{e\in \E_h\setminus\Gamma_d} h_T \|\ljump a\nabla_w \lambda \cdot\bn\rjump\|_e^2 + s(\lambda , \lambda)\Big)^{\frac{1}{2}}. \end{equation}

\begin{lemma}\label{lemmanorm}
The semi-norm $\3bar \cdot \3bar_{\Gamma_d }$ defines a norm in the linear space $V^h_{0,\Gamma_d}$. Likewise, the semi-norm $\3bar \cdot \3bar_{\Gamma_n^c }$ defines a norm in the linear space $V^h_{0,\Gamma_n^c}$.
\end{lemma}

\begin{proof}
We only need to verify the positivity property of $\3bar \cdot\3bar_{ \Gamma_d}$. To this end, assume $\3bar v\3bar_{ \Gamma_d }=0$ for a given $v =\{v_0, v_b\} \in V^h_{0,\Gamma_d}$. It follows that $ v_0=v_b$ on each $\partial T$, $ \nabla \cdot(a\nabla v_0)=0$ on each element $T\in {\cal T}_h$, and $\ljump a\nabla v_0 \cdot \bn \rjump=0$ on each $e \in \E_h \setminus \Gamma_n^c$. Thus, $v_0\in C^0(\Omega)$ is a strong solution of $\nabla\cdot(a\nabla v_0)=0$ in $\Omega$. From $v_0=v_b$ on each $\partial T$ and $v_b=0$ on $\Gamma_d$, we obtain $v_0=0$ on $\Gamma_d$. Furthermore, from $\ljump a\nabla v_0 \cdot \bn \rjump=0$ on $ \E_h \setminus \Gamma_n^c$, we have $a\nabla v_0\cdot \bn=0$ on $\Gamma_n$. Thus, it follows from Lemma \ref{unilem} that $v_0\equiv 0$ in $\Omega$, which leads to $v_b\equiv 0$ on each $\partial T $ by using  $v_b=v_0$ $\partial T$. This shows that $v\equiv 0$ in $\Omega$. A similar argument can be made to show that $\3bar \cdot \3bar_{ \Gamma_n^c }$ defines a norm in the linear space $V^h_{0,\Gamma_n^c}$. This completes the proof of the lemma.
\end{proof}

On any element $T\in\T_h$, the following trace inequality holds true
\begin{equation}\label{traceinequality}
\|\varphi\|_e^2 \leq C
(h_T^{-1}\|\varphi\|_T^2+h_T\|\nabla\varphi\|_T^2)
\end{equation}
for $\varphi\in H^1(T)$; readers are referred to \cite{wy3655} for a derivation of \eqref{traceinequality} under the shape regularity assumption on the finite element partition $\T_h$. For polynomials $\varphi$ in the element $T\in \T_h$,  it follows from the inverse inequality (see also \cite{wy3655}) that
\begin{equation}\label{tracepolynomial}
\|\varphi\|_e^2 \leq C h_T^{-1}\|\varphi\|_T^2.
\end{equation}
Here $e$ is an edge or flat face on the boundary of $T$.

The following  Lemma  shows that the norms defined in (\ref{normgamma1}) and (\ref{normgamma11})  are indeed equivalent.
\begin{lemma} \label{equiva1}
There exist  $C_1>0$  and $C_2>0$ such that
\begin{equation}\label{equi}
C_1 \3bar \lambda  \3bar _{h,  \Gamma_n^c }\leq  \3bar \lambda  \3bar_{ \Gamma_n^c }\leq C_2  \3bar \lambda  \3bar _{h, \Gamma_n^c },
\end{equation}
 for all $\lambda \in V^h_{0,\Gamma_n^c}$.
\end{lemma}
\begin{proof}
First of all, from (\ref{disgradient*}) we have
$$
(\nabla_w \lambda- \nabla\lambda_0 , \boldsymbol{\phi})_T=   \langle \lambda_b- \lambda_0, \boldsymbol{\phi}\cdot \bn\rangle_{\partial T}
$$
for all $\boldsymbol{\phi}\in [P_{k-1}(T)]^2$. Thus, by using the Cauchy-Schwarz  inequality and the trace inequality (\ref{tracepolynomial}) we have
\begin{equation*}
\begin{split}
|(\nabla_w \lambda- \nabla\lambda_0 , \boldsymbol{\phi})_T|
 =&|\langle \lambda_b- \lambda_0, \boldsymbol{\phi}\cdot \bn\rangle_{\partial T}|
\\ \leq &\| \lambda_b- \lambda_0\|_{\partial T} \|\boldsymbol{\phi}\cdot \bn\|_{\partial T}
\\ \leq&Ch_T^{-\frac{1}{2}}\| \lambda_b- \lambda_0\|_{\partial T} \|\boldsymbol{\phi} \|_{T},
\end{split}
\end{equation*}
which leads to
$$
\|\nabla_w \lambda- \nabla\lambda_0\|_T\leq Ch_T^{-\frac{1}{2}}\| \lambda_b- \lambda_0\|_{\partial T}.
$$
Hence,
\begin{equation}\label{esti}
\sum_{T\in {\cal T}_h}\|\nabla_w \lambda- \nabla\lambda_0\|^2_T\leq C \sum_{T\in {\cal T}_h}h_T^{-1} \| \lambda_b- \lambda_0\|^2_{\partial T}\leq Cs(\lambda, \lambda).
\end{equation}

Next, using the triangle inequality, the trace inequality  (\ref{tracepolynomial}), and  (\ref{esti}), we obtain
\begin{equation}\label{term1}
 \begin{split}
&h\sum_{e\in {\cal E}_h\setminus\Gamma_d} \|\ljump a\nabla_w\lambda\cdot \bn\rjump\|_e^2\\
\leq&   h\sum_{e\in {\cal E}_h\setminus\Gamma_d} \|\ljump a(\nabla_w\lambda -\nabla \lambda_0)\cdot \bn\rjump\|_e^2+h \sum_{e\in {\cal E}_h\setminus\Gamma_d}  \| \ljump   a\nabla \lambda_0 \cdot \bn\rjump\|_e^2\\
 \leq &h \sum_{T\in {\cal T}_h} \| a(\nabla_w\lambda -\nabla \lambda_0)\cdot \bn \|_{\partial T}^2+ h\sum_{e\in {\cal E}_h\setminus\Gamma_d}  \| \ljump  a \nabla \lambda_0 \cdot \bn\rjump\|_e^2\\
\leq &    \sum_{T\in {\cal T}_h}  \| \nabla_w \lambda-\nabla \lambda_0 \|^2_{T}+  h\sum_{e\in {\cal E}_h\setminus\Gamma_d}  \| \ljump  a \nabla \lambda_0 \cdot \bn\rjump\|_e^2\\
\leq & C  s(\lambda, \lambda)+h\sum_{e\in {\cal E}_h\setminus\Gamma_d}  \| \ljump   a\nabla \lambda_0 \cdot \bn\rjump\|_e^2.\\
\end{split}
\end{equation}

We now estimate the term $\|\nabla\cdot (a\nabla_w \lambda)\|_T^2$.  Note that $\nabla_w \lambda \in [P_{k-1}(T)]^2$ so that  $\nabla\cdot (a\nabla_w \lambda) \in P_{k-2}(T)$. For any $\psi\in P_{k-2}(T)$, using the usual integration by parts  and (\ref{disgradient*}) we get
\begin{equation*}
\begin{split}
& \quad (\nabla\cdot (a\nabla_w \lambda), \psi )_T\\
&= -(a\nabla_w\lambda, \nabla \psi)_T+\langle a\nabla_w \lambda\cdot \bn, \psi\rangle_{\partial T}\\
&= -(a\nabla \lambda_0, \nabla \psi)_T+\langle \lambda_0-\lambda_b, a\nabla \psi\cdot \bn\rangle_{\partial T}+\langle a\nabla_w \lambda\cdot \bn, \psi\rangle_{\partial T}\\
&= (\nabla\cdot(a\nabla \lambda_0), \psi)_T-\langle a\nabla \lambda_0\cdot \bn, \psi\rangle_{\partial T}+ \langle \lambda_0-\lambda_b, a\nabla \psi\cdot \bn\rangle_{\partial T}+ \langle a\nabla_w \lambda\cdot \bn, \psi\rangle_{\partial T}\\
&= (\nabla\cdot(a\nabla \lambda_0),\psi)_T+\langle a(\nabla_w \lambda-\nabla \lambda_0)\cdot \bn, \psi\rangle_{\partial T}+ \langle  \lambda_0-\lambda_b, a\nabla \psi\cdot \bn\rangle_{\partial T} ,\\
\end{split}
\end{equation*}
which implies
$$
(\nabla\cdot (a\nabla_w \lambda)- \nabla\cdot(a\nabla \lambda_0), \psi )_T= \langle a(\nabla_w \lambda-\nabla \lambda_0)\cdot \bn, \psi\rangle_{\partial T}+ \langle \lambda_0-\lambda_b, a\nabla \psi\cdot \bn\rangle_{\partial T}.
$$
Thus, using the Cauchy-Schwarz inequality, the inverse inequality and the trace inequality (\ref{tracepolynomial}), we arrive at
\begin{equation*}
\begin{split}
& | (\nabla\cdot (a\nabla_w \lambda)- \nabla\cdot(a\nabla \lambda_0), \psi )_T|\\
\leq &\| a(\nabla_w \lambda-\nabla \lambda_0)\cdot \bn\|_{\partial T} \| \psi\|_{\partial T}+ \| \lambda_0-\lambda_b\|_{\partial T}\| a\nabla \psi \|_{\partial T} \\
  \leq &Ch_T^{-1}\| \nabla_w \lambda-\nabla \lambda_0 \|_{  T} \| \psi\|_{  T}+ Ch_T^{-\frac{3}{2}}\|  \lambda_0-\lambda_b\|_{\partial T}\|   \psi \|_{ T}.
\end{split}
\end{equation*}
Therefore, we obtain
$$
\|\nabla\cdot (a\nabla_w \lambda)- \nabla\cdot(a\nabla \lambda_0)\|_T^2 \leq Ch_T^{-2}\|  \nabla_w \lambda-\nabla \lambda_0 \|^2_{  T}+ Ch_T^{-3}\|  \lambda_0-\lambda_b\|^2_{\partial T}.
$$
Using (\ref{esti}) we have
\begin{equation*}
\begin{split}
&\sum_{T\in {\cal T}_h} \|\nabla\cdot (a\nabla_w \lambda)-\nabla\cdot(a\nabla \lambda_0)\|_T^2\\
\leq &\sum_{T\in {\cal T}_h}  Ch_T^{-2}\|  \nabla_w \lambda-\nabla \lambda_0 \|^2_{T}+Ch_T^{-3}\|  \lambda_0-\lambda_b\|_{\partial T}^2\\
\leq &Ch^{-2}s(\lambda,\lambda),
\end{split}
\end{equation*}
or equivalently,
$$
h^2 \sum_{T\in {\cal T}_h} \|\nabla\cdot (a\nabla_w \lambda)-\nabla\cdot(a\nabla \lambda_0)\|_T^2 \leq Cs(\lambda,\lambda).
$$
Combining the above estimate with the triangle inequality gives
\begin{equation}\label{82}
 \begin{split}
& h^2 \sum_{T\in {\cal T}_h} \|\nabla\cdot (a\nabla_w \lambda) \|_T^2 \\
\leq & h^2 \sum_{T\in {\cal T}_h} \|\nabla\cdot (a\nabla  \lambda_0) \|_T^2 +h^2 \sum_{T\in {\cal T}_h}\|\nabla\cdot (a\nabla_w \lambda)-\nabla\cdot(a\nabla \lambda_0)\|_T^2 \\
\leq & h^2 \sum_{T\in {\cal T}_h} \|\nabla\cdot (a\nabla  \lambda_0) \|_T^2 +  Cs(\lambda,\lambda).
\end{split}
\end{equation}

Now it follows from (\ref{term1}) and (\ref{82}) that there exists a constant $C_1$ such that
$$
\3bar \lambda\3bar_{\Gamma_n^c} \geq C_1 \3bar \lambda\3bar_{h, \Gamma_n^c}.
$$
This gives the lower-bound  estimate of $\3bar \lambda\3bar_{\Gamma_n^c}$ in (\ref{equi}).  The upper-bound estimate of $\3bar \lambda\3bar_{\Gamma_n^c}$ can be established analogously, but with details omitted.
\end{proof}

Similar to Lemma \ref{equiva1}, the following lemma shows that the norms defined in (\ref{norm}) and (\ref{norm1}) are also equivalent.
\begin{lemma}\label{equiva2}
There exist  $C_1>0$  and $C_2>0$ such that
$$
C_1 \3bar v  \3bar_{h, \Gamma_d}\leq  \3bar v \3bar_{ \Gamma_d }\leq C_2  \3bar v\3bar_{h, \Gamma_d},
$$
for all $v \in V^h_{0,\Gamma_d}$.
 \end{lemma}

\begin{lemma}\label{lem3}
(generalized inf-sup condition) For any  $\lambda \in  V_{0,\Gamma_n^c}^h$,  there exists a $v\in V_{0,\Gamma_d}^h$, such that
\begin{align}\label{inf1}
 b(v,\lambda) \geq& C_1  \3bar  \lambda\3bar_{h, \Gamma_n^c } ^2 -C_2s(\lambda, \lambda),\\
s(v,v) \leq & C \3bar   \lambda\3bar^2_{h,  \Gamma_n^c }. \label{inf2}
\end{align}
\end{lemma}

\begin{proof} Using (\ref{disgradient}) with $\boldsymbol{\psi}=a\nabla_w \lambda$, we get
\begin{equation}\label{bfor}
\begin{split}
b(v,\lambda) =& \sum_{T\in {\cal T}_h} (a\nabla_w v, \nabla_w \lambda)_T\\
=& \sum_{T\in {\cal T}_h} (v_0, -\nabla \cdot (a\nabla_w \lambda))_T+\langle v_b, a\nabla_w \lambda \cdot \bn\rangle_{\partial T}\\
=& \sum_{T\in {\cal T}_h} (v_0, -\nabla \cdot (a\nabla_w \lambda))_T+\sum_{e\in\E_h\setminus\Gamma_d}\langle v_b, \ljump a\nabla_w \lambda \cdot \bn\rjump\rangle_{e},
\end{split}
\end{equation}
where we have used the fact that $v_b=0$ on $\Gamma_d$.

By letting $v\in V^h_{0,\Gamma_d}$ such that $v_0=-h_T^2 \nabla\cdot (a\nabla_w \lambda)$ on each element $T\in {\cal T}_h$ and $v_b= h_T \ljump a\nabla_w \lambda \cdot \bn\rjump$ on  $e\in\E_h\setminus\Gamma_d$ in (\ref{bfor}),  we arrive at
\begin{equation*}\label{bfor2}
  b(v,\lambda)=   \sum_{T\in {\cal T}_h}h_T^2\| \nabla\cdot (a\nabla_w \lambda)\|_T^2+  \sum_{e\in\E_h\setminus\Gamma_d}h_T  \|\ljump a\nabla_w \lambda \cdot \bn\rjump\|_e^2,
 \end{equation*}
 which, together with the definition (\ref{normgamma11}) of the norm $\3bar\cdot\3bar_{h,\Gamma_n^c}$, completes the proof of (\ref{inf1}).

As to (\ref{inf2}),  note that $v\in V^h_{0,\Gamma_d}$ is chosen such that $v_0=-h_T^2 \nabla\cdot (a\nabla_w \lambda)$ on each element $T\in {\cal T}_h$ and $v_b= h_T \ljump a\nabla_w \lambda \cdot \bn\rjump$ on each  $e\in\E_h\setminus \Gamma_d$. Thus, from the trace inequality (\ref{tracepolynomial}),  the triangle inequality and   (\ref{normgamma11}), we have
\begin{equation*}
\begin{split}
s(v,v)=&\sum_{T\in {\cal T}_h}h_T^{-1} \| v_0-v_b\|_{\partial T}^2 \\
& \leq \sum_{T\in {\cal T}_h} h_T^{-1} \| v_0\|_{\partial T}^2+ \sum_{e\in\E_h\setminus \Gamma_d}h_T^{-1} \|v_b\|_{e}^2\\
& \leq \sum_{T\in {\cal T}_h} h_T^{-1} \|h_T^2 \nabla\cdot (a\nabla_w \lambda) \|_{\partial T}^2+ \sum_{e\in\E_h\setminus \Gamma_d}  h_T^{-1}
\|h_T \ljump a\nabla_w \lambda \cdot \bn\rjump\|_{e}^2\\
& \leq \sum_{T\in {\cal T}_h} h_T^{2} \|\nabla\cdot (a\nabla_w \lambda)  \|_{ T}^2+  \sum_{e\in\E_h\setminus\Gamma_d} h_T
\|\ljump a\nabla_w \lambda \cdot \bn\rjump\|_{e}^2\\
& \leq  C \3bar   \lambda\3bar^2_{h,  \Gamma_n^c },
\end{split}
\end{equation*}
which completes the proof of (\ref{inf2}).
\end{proof}

Similar to Lemma \ref{lem3}, we have the following result:
\begin{lemma}\label{lem4}
(generalized inf-sup condition)  For any $v \in  V_{0,\Gamma_d}^h$, there exists a $ \lambda \in V_{0,\Gamma_n^c}^h$, such that
\begin{align}\label{inf3}
b(v,\lambda) &\geq
C_1  \3bar  v \3bar_{h,  \Gamma_d } ^2 -C_2 s(v, v),\\
s(\lambda,\lambda) & \leq  C \3bar    v\3bar^2_{h,  \Gamma_d}. \label{inf4}
\end{align}
\end{lemma}

\begin{proof} Using (\ref{disgradient}) with $\boldsymbol{\psi}=a\nabla_w \lambda$, we get
\begin{equation}\label{bfor3}
\begin{split}
b(v,\lambda) =& \sum_{T\in {\cal T}_h} (\nabla_w v, a\nabla_w \lambda)_T\\
=& \sum_{T\in {\cal T}_h} ( \lambda_0, -\nabla \cdot (a\nabla_w v))_T+\langle  \lambda_b, a\nabla_w v \cdot \bn\rangle_{\partial T}\\
=& \sum_{T\in {\cal T}_h} ( \lambda_0, -\nabla \cdot (a\nabla_w v))_T+\sum_{e\in\E_h\setminus\Gamma_n^c}\langle  \lambda_b, \ljump a\nabla_w v \cdot \bn\rjump\rangle_{e},
\end{split}
\end{equation}
where we have used the fact that $\lambda_b=0$ on $\Gamma_n^c$.

By letting  $\lambda \in V_{0,\Gamma_n^c}^h$ such that $ \lambda_0=-h_T^2 \nabla\cdot (a\nabla_w v)$ on each element $T\in {\cal T}_h$ and $ \lambda_b= h_T \ljump a\nabla_w v \cdot \bn\rjump$ on  $e\in\E_h\setminus\Gamma_n^c$ in (\ref{bfor3}),   we arrive at
\begin{equation*}\label{bfor4}
b(v,\lambda)=  \sum_{T\in {\cal T}_h}h_T^2 \| \nabla\cdot (a\nabla_w v)\|_T^2+ \sum_{e\in\E_h\setminus\Gamma_n^c}  h_T\|\ljump a\nabla_w v \cdot \bn\rjump\|_e^2,
\end{equation*}
which, together with the definition (\ref{norm}) of the norm $\3bar\cdot\3bar_{h,  \Gamma_d}$, completes the proof of (\ref{inf3}).

As to (\ref{inf4}), note that  $\lambda \in V_{0,\Gamma_n^c}^h$ is given by $ \lambda_0=-h_T^2 \nabla\cdot (a\nabla_w v)$ on each element $T\in {\cal T}_h$ and $ \lambda_b= h_T \ljump a\nabla_w v \cdot \bn\rjump$ on  $e\in\E_h\setminus\Gamma_n^c$. Thus, from the trace inequality  (\ref{tracepolynomial}), the triangle inequality, and   (\ref{norm}), we have
\begin{equation*}
\begin{split}
s(\lambda,\lambda)=&\sum_{T\in {\cal T}_h}h_T^{-1} \| \lambda_0-\lambda_b\|_{\partial T}^2 \\
& \leq \sum_{T\in {\cal T}_h} h_T^{-1} \| \lambda_0\|_{\partial T}^2+ \sum_{e\in\E_h\setminus\Gamma_n^c} h_T^{-1} \|\lambda_b\|_{e}^2\\
& \leq \sum_{T\in {\cal T}_h} h_T^{-1} \|h_T^2 \nabla\cdot (a\nabla_w v) \|_{\partial T}^2+  \sum_{e\in\E_h\setminus\Gamma_n^c } h_T^{-1}\|h_T  \ljump a\nabla_w v \cdot \bn\rjump\|_{e}^2\\
& \leq \sum_{T\in {\cal T}_h} h_T^{2} \|   \nabla\cdot (a\nabla_w v)  \|_{ T}^2+  \sum_{e\in\E_h\setminus\Gamma_n^c} h_T\|  \ljump a\nabla_w v \cdot \bn\rjump\|_{e}^2\\
& \leq  C \3bar   v \3bar^2_{h,  \Gamma_d },
\end{split}
\end{equation*}
which completes the proof of (\ref{inf4}).
\end{proof}

\section{Error Equations}\label{Section:EE}
We shall derive an error equation for the primal-dual weak Galerkin algorithm (\ref{32})-(\ref{2}). To this end, let $u$ and $(u_h;\lambda_h)  \in V^h\times V^h_{0,\Gamma_n^c}$ be the solutions of the model problem (\ref{1}) and  the primal-dual weak Galerkin algorithm (\ref{32})-(\ref{2}), respectively. Note that the exact solution  of the Lagrange multiplier $\lambda$ is 0. Define the error functions by
\begin{align}\label{error}
e_h&=u_h-Q_h u,
\\
\epsilon_h&=\lambda_h-Q_h\lambda=\lambda_h.\label{error-2}
\end{align}

\begin{lemma}\label{errorequa}
Let $u$ and $(u_h;\lambda_h) \in V^h\times V^h_{0,\Gamma_n^c}$ be the solutions of the model problem (\ref{1}) and  the primal-dual weak Galerkin algorithm (\ref{32})-(\ref{2}), respectively.  Then, the error functions $e_h$ and $\epsilon_h$ defined in (\ref{error})-(\ref{error-2}) satisfy the following error equations
\begin{eqnarray}\label{sehv}
s( e_h, v)-b(v, \epsilon_h) &=&-s(Q_h u, v),\qquad \forall v\in V^h_{0,\Gamma_d},\\
s ( \epsilon_h, w)+b(e_h, w)&=&\ell_u(w),\qquad\qquad \forall w\in V^h_{0,\Gamma_n^c}. \label{sehv2}
\end{eqnarray}
\end{lemma}
Here, \begin{equation}\label{lu}
\begin{split}
\qquad \ell_u(w) = \sum_{T\in {\cal T}_h} \langle a\nabla u\cdot \bn-{\cal Q}_h (a\nabla u)\cdot \bn, w_b-w_0\rangle_{\partial T}.
\end{split}
\end{equation}

\begin{proof}
Subtracting $s( Q_hu, v)$ from both sides of (\ref{32}) yields
\begin{equation*}
\begin{split}
s( u_h-Q_h u, v)-b(v, \lambda_h) =-s(Q_h u, v).
\end{split}
\end{equation*}
This gives
$$
s( e_h, v)-b(v, \epsilon_h) =-s(Q_h u, v),
$$
which completes the proof of (\ref{sehv}).

Next, by subtracting $b( Q_hu, w)$ from both sides of (\ref{2}) we have for any $w\in V_{0,\Gamma_n^c}^h$ that
\begin{equation}\label{err2}
\begin{split}
&s (\lambda_h-Q_h \lambda, w)+b(u_h-Q_hu, w)\\
= &(f,w_0)+\langle g_2,w_b\rangle_{\Gamma_n}-b( Q_hu, w)\\
=&(-\nabla\cdot(a\nabla u),w_0)+\langle a\nabla u\cdot \bn,w_b\rangle_{\Gamma_n}- \sum_{T\in {\cal T}_h}( a\nabla_w(Q_hu), \nabla_w w)_T\\
=&(-\nabla\cdot(a\nabla u),w_0)+\langle a\nabla u\cdot \bn,w_b\rangle_{\Gamma_n}-
\sum_{T\in {\cal T}_h}( a{\cal Q}_h (\nabla u), \nabla _w w)_T, \\
\end{split}
\end{equation}
where we used Lemma \ref{Lemma5.1} in the last line.

Now, by letting $\boldsymbol{\psi}=a{\cal Q}_h (\nabla u)$ in (\ref{disgradient*}) and using the usual integration by parts, we obtain
\begin{equation}\label{err1}
\begin{split}
&\sum_{T\in {\cal T}_h}( a{\cal Q}_h (\nabla u), \nabla _w w)_T\\
=& \sum_{T\in {\cal T}_h}( a{\cal Q}_h (\nabla u), \nabla  w_0)_T+\sum_{T\in {\cal T}_h} \langle a{\cal Q}_h (\nabla u)\cdot \bn, w_b-w_0 \rangle_{\partial T}\\
=&\sum_{T\in {\cal T}_h} (a\nabla u, \nabla  w_0)_T+\sum_{T\in {\cal T}_h} \langle a{\cal Q}_h (\nabla u)\cdot \bn, w_b-w_0 \rangle_{\partial T}\\
=&\sum_{T\in {\cal T}_h}- (\nabla\cdot(a\nabla u), w_0)_T+ \langle a\nabla u\cdot \bn, w_0\rangle_{\partial T}+  \langle a{\cal Q}_h (\nabla u)\cdot \bn, w_b-w_0 \rangle_{\partial T}\\
 =& \sum_{T\in {\cal T}_h}- (\nabla\cdot(a\nabla u), w_0)_T+  \langle a\nabla u\cdot \bn, w_0-w_b\rangle_{\partial T}+   \langle a{\cal Q}_h (\nabla u)\cdot \bn, w_b-w_0 \rangle_{\partial T}\\&+\sum_{e\subset \Gamma_n}  \langle  a \nabla u \cdot \bn, w_b  \rangle_{e}  \\
 =& \sum_{T\in {\cal T}_h}- (\nabla\cdot(a\nabla u), w_0)_T+ \langle a\nabla u\cdot \bn-a{\cal Q}_h (\nabla u)\cdot \bn, w_0-w_b\rangle_{\partial T}\\
 & +\sum_{e\subset \Gamma_n}  \langle   a\nabla u \cdot \bn, w_b  \rangle_{e},\\
 \end{split}
\end{equation}
where we used $w_b=0$ on $\Gamma_n^c$ on the sixth line. Substituting (\ref{err1}) into (\ref{err2}) completes the proof of (\ref{sehv2}).
\end{proof}

\section{Error Estimates in a Scaled Residual Norm}\label{Section:ErrorEstimate}
The goal of this section is to derive an error estimate for the solution of the primal-dual weak Galerkin algorithm (\ref{32})-(\ref{2}). First of all, let us recall the following error estimates for the $L^2$ projection operator.

%

\begin{lemma}\cite{wy3655}
Let ${\cal T}_h$ be a finite element  partition of $\Omega$ satisfying the shape regular assumption given in \cite{wy3655}. Then, for any $0\leq s \leq 2$ and $1\leq m \leq k$, one has
\begin{eqnarray}\label{error1}
\sum_{T\in {\cal T}_h}h_T^{2s}\|u-Q_0u\|^2_{s,T}&\leq& C h^{2(m+1)}\|u\|^2_{m+1},\\
\label{error3}
\sum_{T\in {\cal T}_h}h_T^{2s}\|u-{\cal Q}_hu\|^2_{s,T}&\leq& C h^{2m}\|u\|^2_{m}.
\end{eqnarray}
\end{lemma}

Using Lemma \ref{error1} and the error equation derived in the previous section we arrive at the following result:

\begin{theorem} \label{theoestimate}
For $k\geq 1$, let $(u_h;\lambda_h)  \in V^h\times V^h_{0,\Gamma_n^c}$ be the numerical approximation of the elliptic Cauchy problem (\ref{1}) obtained from the primal-dual weak Galerkin algorithm (\ref{32})-(\ref{2}). Assume that the exact solution $u$ of (\ref{1}) is sufficiently regular such that $u\in H^{k+1}(\Omega)$. Let the error functions  $e_h$ and $ \epsilon_h$ be given in (\ref{error}) and (\ref{error-2}). Then, the following error estimate holds true:
\begin{equation}\label{erres}
\3bar e_h \3bar_{h,\Gamma_d}+\3bar \epsilon_h\3bar_{h,\Gamma_n^c} \leq Ch^{k } \|u\|_{k+1}.
\end{equation}
\end{theorem}

\begin{proof} It is easy to verify that $e_h \in V^h_{0,\Gamma_d}$ and $\epsilon_h \in V^h_{0,\Gamma_n^c}$. By letting $v=e_h$ in (\ref{sehv}) and $w=\epsilon_h$ in (\ref{sehv2}) and then summing (\ref{sehv}) with (\ref{sehv2})  we obtain
\begin{equation}\label{ess2}
s(e_h,e_h)+s(\epsilon_h, \epsilon_h)=-s(Q_hu, e_h)+ \ell_u(\epsilon_h).
\end{equation}
Now from the Cauchy Schwarz inequality, the trace inequality (\ref{traceinequality}), (\ref{lu}), and the estimate (\ref{error3}) we have
\begin{equation}\label{est1}
 \begin{split}
&\qquad|\ell_u(\epsilon_h)|\\
&= \Big| \sum_{T\in {\cal T}_h}   \langle a\nabla u\cdot \bn-{\cal Q}_h (a\nabla u)\cdot \bn, \epsilon_b-\epsilon_0\rangle_{\partial T} \Big| \\
&\leq \Big (\sum_{T\in {\cal T}_h}h_T \| a\nabla u\cdot \bn-{\cal Q}_h (a\nabla u)\cdot \bn\|_{\partial T}^2\Big)^{\frac{1}{2}}
\Big ( \sum_{T\in {\cal T}_h} h_T^{-1}  \| \epsilon_b-\epsilon_0\|_{\partial T}^2\Big)^{\frac{1}{2}}\\
&\leq C \Big ( \sum_{T\in {\cal T}_h}  \| \nabla u-{\cal Q}_h (\nabla u)\|_{T}^2+h_T^2  \| \nabla u -{\cal Q}_h (\nabla u)\|_{1,T}^2\Big)^{\frac{1}{2}}
s( \epsilon_h, \epsilon_h )^{\frac{1}{2}}\\
&\leq Ch^k\|u\|_{k+1}s(\epsilon_h, \epsilon_h )^{\frac{1}{2}}.
\end{split}
\end{equation}
Next, we use the Cauchy-Schwarz inequality, the trace inequality (\ref{traceinequality}), and the estimate (\ref{error1}) to obtain
\begin{equation}\label{est2}
\begin{split}
\Big| s(Q_hu, e_h)\Big| & = \Big| \sum_{T\in {\cal T}_h} h_T^{-1} \langle Q_0u-Q_bu, e_0-e_b\rangle_{\partial T}\Big| \\
& \leq \Big ( \sum_{T\in {\cal T}_h}h_T^{-1} \| Q_0u-u\|_{\partial T}^2\Big)^{\frac12}
\Big ( \sum_{T\in {\cal T}_h}h_T^{-1} \|e_0-e_b\|_{\partial T}^2\Big)^{\frac12}\\
& \leq C \Big ( \sum_{T\in {\cal T}_h}h_T^{-1} \| Q_0u- u\|_{\partial T}^2\Big)^{\frac12}
s( e_h, e_h )^{\frac{1}{2}}\\
& \leq C \Big ( \sum_{T\in {\cal T}_h}h_T^{-2} \| Q_0u- u\|_{  T}^2+  \| Q_0u- u\|_{ 1, T}^2\Big)^{\frac12} s( e_h, e_h )^{\frac{1}{2}}\\
&\leq Ch^k\|u\|_{k+1}s(e_h, e_h)^{\frac{1}{2}}.
\end{split}
\end{equation}
Combining (\ref{ess2}) with (\ref{est1}) and (\ref{est2}) gives rise to
\begin{equation*}
s(e_h,e_h)+s(\epsilon_h, \epsilon_h) \leq  Ch^{k}\|u\|_{k+1}(s(\epsilon_h, \epsilon_h )^{\frac{1}{2}} + s(e_h, e_h)^{\frac{1}{2}}),
\end{equation*}
which leads to
\begin{equation}\label{serror}
s(e_h,e_h)+s(\epsilon_h, \epsilon_h) \leq  Ch^{2k}\|u\|_{k+1}^2.
\end{equation}

Next, from (\ref{sehv2}) we have
\begin{equation}\label{est3}
b(e_h, w)=\ell_u(w)-s(\epsilon_h, w)\qquad \forall w\in V_{0,\Gamma_n^c}^h.
\end{equation}
Using the generalized inf-sup condition in Lemma \ref{lem4}, there exists a $w \in V_{0,\Gamma_n^c}^h$ such that
\begin{equation}\label{est4}
b(e_h,w)  \geq C_1  \3bar e_h \3bar_{h,  \Gamma_d }^2 -C_2s(e_h, e_h).
\end{equation}
Thus, with this particular $w$, we have from (\ref{est3}) and (\ref{est4}) the following estimate:
\begin{equation}\label{eherror}
\3bar e_h \3bar_{h,  \Gamma_d } ^2  \leq  |\ell_u(w)|+|s(\epsilon_h, w)|+ C_2 s(e_h, e_h).
\end{equation}

Now from  (\ref{est1}) we have
$$
| \ell_u(w) | \leq Ch^{k}\|u\|_{k+1} s(w,w)^{\frac{1}{2}},
$$
and from the Cauchy-Schwarz inequality,
\begin{equation*}
|s(\epsilon_h, w)| \leq s(\epsilon_h, \epsilon_h)^{\frac{1}{2}} s(w,w)^{\frac{1}{2}}.
\end{equation*}
Substituting the above two inequalities into (\ref{eherror}) yields
\begin{equation*}
\3bar e_h \3bar_{h,  \Gamma_d } ^2  \leq   Ch^{k}\|u\|_{k+1} s(w,w)^{\frac{1}{2}} + s(\epsilon_h, \epsilon_h)^{\frac{1}{2}} s(w,w)^{\frac{1}{2}} + C_2 s(e_h, e_h).
\end{equation*}
Thus, it follows from the estimate (\ref{serror}) that
\begin{equation}\label{esti1}
\3bar e_h \3bar_{h,  \Gamma_d } ^2  \leq   Ch^{k}\|u\|_{k+1} s(w,w)^{\frac{1}{2}} + Ch^{2k}\|u\|_{k+1}^2.
\end{equation}
Recall that, from (\ref{inf4}), this particular $w$ satisfies
\begin{equation}\label{esti2}
s(w,w)^{\frac{1}{2}}\leq   C\3bar e_h\3bar_{h,\Gamma_d}.
\end{equation}
Substituting the above into (\ref{esti1}) gives
\begin{equation*}
\3bar e_h \3bar_{h,  \Gamma_d } ^2  \leq   Ch^{k}\|u\|_{k+1} \3bar e_h\3bar_{h,\Gamma_d} + Ch^{2k}\|u\|_{k+1}^2,
\end{equation*}
which leads to the following
\begin{equation}\label{esti3}
\3bar e_h \3bar_{h,\Gamma_d } \leq Ch^{k}\|u\|_{k+1}.
\end{equation}

As to the estimate of $\3bar  \epsilon_h \3bar_{h,\Gamma_n^c}$, we have from (\ref{sehv}) that
\begin{equation}\label{est5}
b(v, \epsilon_h)=s(e_h,v)+s(Q_hu, v)\qquad \forall v\in V_{0,\Gamma_d}^h.
\end{equation}
From Lemma \ref{lem3}, there exists a particular $v\in V_{0,\Gamma_d}^h$ such that
\begin{equation}\label{est6}
b(v, \epsilon_h) \geq C_1  \3bar \epsilon_h \3bar_{h,  \Gamma_n^c } ^2 -C_2s(\epsilon_h, \epsilon_h).
\end{equation}
Combining (\ref{est5}) with (\ref{est6})  gives
\begin{equation}\label{eherror2}
C_1 \3bar \epsilon_h \3bar_{h,  \Gamma_n^c } ^2  \leq  s(e_h,v)+s(Q_hu, v)+C_2s(\epsilon_h, \epsilon_h).
\end{equation}
From (\ref{est2}) we have
$$
|s(Q_hu, v)|\leq Ch^{k}\|u\|_{k+1} s(v,v)^{\frac{1}{2}},
$$
and from the Cauchy-Schwarz inequality,
$$
|s(e_h,v)| \leq  s(e_h, e_h)^{\frac{1}{2}} s(v,v)^{\frac{1}{2}}.
$$
Combining the above two inequalities with (\ref{serror}) and (\ref{eherror2}) yields
\begin{equation}\label{esti4}
\3bar \epsilon_h \3bar_{h,  \Gamma_n^c }^2 \leq C h^{k} \|u\|_{k+1} s(v,v)^{\frac{1}{2}} + Ch^{2k}\|u\|_{k+1}^2.
\end{equation}
Notice that from (\ref{inf2}), the following estimate holds true for this particular $v$:
\begin{equation}\label{esti5}
s(v,v)^{\frac{1}{2}}\leq   C \3bar \epsilon_h\3bar_{h,  \Gamma_n^c}.
\end{equation}
Thus, substituting (\ref{esti5}) into (\ref{esti4}) gives
\begin{equation}\label{esti6}
\3bar  \epsilon_h \3bar_{h,  \Gamma_n^c}    \leq   Ch^{k}\|u\|_{k+1}.
\end{equation}

Finally, the theorem is proved by combining (\ref{esti3}) with (\ref{esti6}).
\end{proof}

\section{Error Estimate in a Weak $L^2$ Topology}\label{Section:L2Error} This section is devoted to the establishment of an error estimate for the weak Galerkin finite element solution $u_h$ in a weak $L^2$ topology. To this end, consider the auxiliary problem which seeks an unknown function $\Phi$ satisfying
\begin{equation}\label{dual1}
\left\{
\begin{split}
-\nabla\cdot(a\nabla \Phi)=& \eta,\qquad\  \text{in}\quad \Omega,\\
\Phi=& \ 0,\qquad \text{on}\quad \Gamma_n^c, \\ 
a\nabla\Phi \cdot \bn=& \ 0,\qquad \text{on}\quad \Gamma_d^c,\\ 
\end{split}\right.
\end{equation}
where $\eta\in L^2(\Omega)$ and $\Gamma_d^c=\partial\Omega\setminus\Gamma_d$. Denote by $X_\gamma$ the set of all functions $\eta\in L^2(\Omega)$ so that the problem (\ref{dual1}) has a solution with the $H^{1+\gamma}$-regularity in the sense that
\begin{equation}\label{regul}
\|\Phi\|_{1+ \gamma}\leq C\|\eta\|,
\end{equation}
where $\gamma \in (\frac{1}{2}, 1]$ is a parameter.

\begin{theorem}\label{Thm:L2errorestimate}
For $k\geq 1$, let $(u_h;\lambda_h) \in V^h\times V^h_{0,\Gamma_n^c}$ be the solution of the primal-dual weak Galerkin equations (\ref{32})-(\ref{2}). Assume that the exact solution is sufficiently regular such that $u\in H^{k+1}(\Omega)$. Then, there exists a constant $C$ such that
\begin{equation}\label{e0}
\sup_{\eta\in X_\gamma} \frac{|(Q_0u - u_0,\eta)|}{\|\eta\|}  \leq Ch^{k+\gamma}\|u\|_{k+1}.
\end{equation}
\end{theorem}

\begin{proof} By testing (\ref{dual1}) with $e_0$ on each element $T\in\T_h$, we obtain from the usual integration by parts that
\begin{equation}\label{2.11}
\begin{split}
(\eta, e_0)= &\sum_{T\in {\cal T}_h}(-\nabla\cdot(a\nabla \Phi),   e_0)_T \\
=& \sum_{T\in {\cal T}_h} (a\nabla \Phi, \nabla e_0)_T -\langle a\nabla\Phi\cdot \bn,   e_0  \rangle_{\partial T} \\
=& \sum_{T\in {\cal T}_h} (a\nabla\Phi, \nabla e_0)_T -\langle a\nabla\Phi\cdot \bn,   e_0-e_b  \rangle_{\partial T}, \\
\end{split}
\end{equation}
where we have used the homogeneous boundary condition in (\ref{dual1}) and the fact that $e_b=0$ on $\Gamma_d$ on the third line.

It follows from (\ref{l}) and (\ref{disgradient*}) with $\boldsymbol{\psi}={\cal Q}_h a\nabla \Phi$ that
\begin{align*}
(a\nabla_w e_h, \nabla_w (Q_h \Phi ) )_T =& (\nabla_w e_h, {\cal Q}_h a\nabla \Phi)_T\\
=& (\nabla e_0, {\cal Q}_h a\nabla\Phi)_T- \langle e_0- e_b, {\cal Q}_h a\nabla   \Phi\cdot \bn \rangle_{\partial T} \\
=& (a\nabla e_0, \nabla \Phi)_T- \langle e_0- e_b, {\cal Q}_h a \nabla \Phi\cdot \bn \rangle_{\partial T},
\end{align*}
which leads to
\begin{equation*}\label{2.13}
\begin{split}
 (a\nabla e_0,  \nabla \Phi )_T=(a\nabla_w e_h, \nabla_w (Q_h \Phi) )_T+  \langle e_0- e_b,  {\cal Q}_h a\nabla\Phi\cdot \bn \rangle_{\partial T},
 \end{split}
\end{equation*}
from which, (\ref{2.11}) can be  rewritten as follows
\begin{equation}\label{2.14}
\begin{split}
&(\eta, e_0)\\
=&\sum_{T\in {\cal T}_h}  (a\nabla_w e_h, \nabla_w (Q_h \Phi ) )_T+  \langle e_0- e_b,  {\cal Q}_h a\nabla\Phi\cdot \bn \rangle_{\partial T} -\langle a\nabla  \Phi\cdot \bn,   e_0-e_b  \rangle_{\partial T}  \\
=& \sum_{T\in {\cal T}_h}  (a\nabla_w e_h, \nabla_w (Q_h \Phi ) )_T+  \langle e_0- e_b,  ({\cal Q}_h a\nabla \Phi -a\nabla \Phi )\cdot \bn \rangle_{\partial T}.\\
\end{split}
\end{equation}

Let us deal with the first term on the last line of (\ref{2.14}). Note that $Q_b\Phi=0$ on $\Gamma_n^c$ due to the Dirichlet boundary condition in (\ref{dual1}). By setting $w:=Q_h \Phi=\{Q_0 \Phi, Q_b \Phi\}$ in the error equation (\ref{sehv2}), and then using the triangle inequality, the Cauchy-Schwarz inequality, the trace inequality (\ref{traceinequality}), (\ref{error1}), (\ref{error3}), (\ref{regul}) and (\ref{erres}), we obtain
\begin{equation}\label{2.14.120:L2}
\begin{split}
& \Big| \sum_{T\in {\cal T}_h}  (a\nabla_w e_h, \nabla_w (Q_h \Phi ) )_T  \Big| \\
= & \Big| \sum_{T\in {\cal T}_h} \langle a\nabla u\cdot \bn-{\cal Q}_h (a\nabla u)\cdot \bn, Q_b \Phi-Q_0\Phi\rangle_{\partial T} - h_T^{-1}\langle  \epsilon_0-\epsilon_b, Q_0 \Phi-Q_b\Phi \rangle_{\partial T}\Big|\\
\leq & \Big( \sum_{T\in {\cal T}_h} h_T\| \nabla u-{\cal Q}_h (\nabla u)\|^2_{\partial T} \Big)^{\frac{1}{2}} \Big( \sum_{T\in {\cal T}_h} h_T^{-1}\| Q_b \Phi-Q_0\Phi \|^2_{\partial T} \Big)^{\frac{1}{2}}\\
&+  \Big( \sum_{T\in {\cal T}_h} h_T^{-1}\|  \epsilon_0-\epsilon_b\|^2_{\partial T} \Big)^{\frac{1}{2}}    \Big( \sum_{T\in {\cal T}_h} h_T^{-1}\| Q_0 \Phi-Q_b\Phi    \|^2_{\partial T} \Big)^{\frac{1}{2}}  \\
\leq &   \Big( \sum_{T\in {\cal T}_h} \| \nabla u-{\cal Q}_h (\nabla u)\|^2_{T} +h_T ^2 \| \nabla u-{\cal Q}_h (\nabla u)\|^2_{1, T} \Big)^{\frac{1}{2}}\\
& \cdot \Big( \sum_{T\in {\cal T}_h}h_T^{-1}  \|\Phi-Q_0\Phi \|^2_{\partial T} \Big)^{\frac{1}{2}}+  \3bar \epsilon_h \3bar_{h,\Gamma_n^c}    \Big( \sum_{T\in {\cal T}_h} h_T^{-1}\| Q_0 \Phi-\Phi \|^2_{\partial T} \Big)^{\frac{1}{2}} \\
\leq & Ch^k \|u\|_{k+1}\Big( \sum_{T\in {\cal T}_h} h_T^{-2}\| Q_0 \Phi-\Phi    \|^2_{T}+\| Q_0 \Phi-\Phi\|^2_{1,T} \Big)^{\frac{1}{2}}\\
\leq & Ch^{k} \|u\|_{k+1} h^{\gamma}\|\Phi\|_{1+\gamma}\\
\leq & Ch^{k+\gamma} \|u\|_{k+1} \|\eta\|.
\end{split}
\end{equation}
For the second  term on the second line of (\ref{2.14}), it follows from the Cauchy-Schwarz inequality,  the trace inequality (\ref{traceinequality}),  (\ref{error3}), (\ref{regul}) and (\ref{erres}) that
\begin{equation}\label{EQ:New:2015:800:L2}
\begin{split}
&|\sum_{T\in{\cal T}_h} \langle e_0- e_b,  ({\cal Q}_h a\nabla\Phi -a\nabla\Phi )\cdot \bn \rangle_{\partial T} | \\
\leq & \Big(\sum_{T\in{\cal T}_h} h_T^{-1} \|e_0- e_b\|^2_{\partial T} \Big)^{\frac{1}{2}} \Big(\sum_{T\in{\cal T}_h} h_T  \|  ({\cal Q}_h a\nabla \Phi -a\nabla \Phi )\cdot \bn\|^2_{\partial T} \Big)^{\frac{1}{2}} \\
\leq & C \3bar e_h\3bar_{h, \Gamma_d}  \Big(\sum_{T\in{\cal T}_h}   \|{\cal Q}_h \nabla\Phi -\nabla\Phi\|^2_{T}+h_T^2\|{\cal Q}_h \nabla\Phi -\nabla\Phi\|^2_{1, T}  \Big)^{\frac{1}{2}}\\
\leq & Ch^{k} \|u\|_{k+1} h^{\gamma}\|\Phi\|_{1+\gamma}\\
\leq & Ch^{k+\gamma} \|u\|_{k+1}\|\eta\|.
\end{split}
\end{equation}

Finally, substituting (\ref{2.14.120:L2}) - (\ref{EQ:New:2015:800:L2}) into (\ref{2.14}) yields
$$
|(\eta, e_0)|  \leq Ch^{k+\gamma}\|u\|_{k+1}\|\eta\|,
$$
which completes the proof of the theorem.
\end{proof}

\section{Numerical Experiments}\label{Section:NE}
In this section, we shall report some numerical results to demonstrate the computational performance of the primal-dual weak Galerkin scheme (\ref{32})-(\ref{2}) for the elliptic Cauchy problem (\ref{1}). The goal is to numerically verify the convergence and stability theory established in the previous sections.

For simplicity, the numerical tests are conducted for the second order elliptic equation with diffusion coeffcient $a=1$ on the unit square domain $\Omega=(0,1)^2$ with uniform triangular partitions. The uniform triangular partitions are obtained by first partitioning the unit square domain $\Omega$ into $n\times n$ uniform sub-squares and then dividing each square element into two triangles by a diagonal line with a negative slope. The numerical tests are implemented for the lowest order element (i.e., $k=1$). The local weak finite element spaces for both the primal variable $u_h$ and the Lagrange multiplier (dual variable) $\lambda_h$ are thus given by
$$
V(1,T)=\{v=\{v_0,v_b \}:\ v_0\in P_1(T), v_b \in P_1(e), \forall \ e\subset \pT\}.
$$
The action of the discrete weak gradient operator on any $v=\{v_0,v_b\}\in V(1,T)$ is computed as a constant-valued vector on $T$ satisfying
\begin{equation*}
(\nabla_{w} v, \bw)_T= - ( v_0,\nabla \cdot \bw)_T +
\langle v_b,\bw \cdot \bn \rangle_{\partial T},\qquad \forall \bw \in [P_0(T)]^2.
\end{equation*}
Since the test function $\bw$ is a constant-valued vector on each element $T\in {\cal T}_h$, the above equation can be simplified as
\begin{equation*}
\begin{split}
(\nabla_{w} v, \bw)_T=  \langle v_b,\bw \cdot \bn \rangle_{\partial T},\qquad \forall  \bw \in [P_0(T)]^2.
\end{split}
\end{equation*}

In the numerical tests, the load function $f=f(x,y)$ and the Cauchy boundary data in the model problem (\ref{1}) are computed by using the given exact solution $u=u(x,y)$. The numerical results are demonstrated for the  error function $e_0=u_0-Q_0u$ measured in the following $L^2$ norm
$$
\|e_0\|_0 =\big( \sum_{T\in {\cal T}_h}\|e_0 \|^2_{T}\big)^{\frac{1}{2}}.
$$
For the error function $e_h=u_h-Q_hu$, we use the following scaled residual norm to measure its magnitude:
$$
\3bar e_h \3bar_{h, \Gamma_d} = \Big(\sum_{T\in {\cal T}_h}  h_T^2\| \nabla\cdot (\nabla_w e_h )\|_T^2 + \sum_{e\in \E_h\setminus \Gamma_n^c} h_T\|\ljump \nabla_w e_h\cdot\bn\rjump\|_e^2+ s(e_h , e_h)\Big)^{\frac{1}{2}}. $$

Tables \ref{NE:TRI:Case1-0} - \ref{NE:TRI:Case1-1} demonstrate the correctness and reliability of the code using the computational results for the elliptic Cauchy problem with the exact solution $u=1+x+y$. Note that the numerical solutions are coincide with the exact solution for this test case. Table \ref{NE:TRI:Case1-0} shows the numerical results for the case when both the Dirichlet and Neumann boundary conditions are set on the horizontal boundary segment $(0,1)\times 0$. Table \ref{NE:TRI:Case1-1} illustrates the performance of the numerical scheme when the vertical boundary segment $0\times (0,1)$ is used to set the Cauchy boundary data. It can be seen from Tables \ref{NE:TRI:Case1-0} - \ref{NE:TRI:Case1-1} that the errors are in machine accuracy, especially for relatively coarse grids. The numerical results are perfectly consistent with the mathematical theory. Tables \ref{NE:TRI:Case1-0} - \ref{NE:TRI:Case1-1} inform us on the correctness of the code for the PD-WG algorithm (\ref{32})-(\ref{2}). However, it should be pointed out that the error seems to deteriorate when the mesh becomes finer and finer. We conjecture that this deterioration might be caused by two factors: (i) the ill-posedness of the elliptic Cauchy problem, and (ii) the poor conditioning of the discrete linear system.

\begin{table}[H]
\begin{center}
 \caption{Numerical error and order of convergence for the exact solution $u=1+x+y$ with Dirichlet and Neumann data set on $(0,1)\times 0$.}\label{NE:TRI:Case1-0}
 \begin{tabular}{|c|c|c|c|}
\hline $1/h$ &   $\| \nabla  e_0\|_{0} $  &$\|e_0\|_{0}$ & $\3bar e_h \3bar_{h, \Gamma_d}$
\\
\hline 1  	&6.44E-15   &2.90E-15&1.08E-14
\\
\hline 2  	& 1.16E-14  & 7.97E-15&1.16E-14
\\
\hline 4 	&2.06E-14   &7.72E-15&1.03E-14
\\
\hline 8 	& 3.95E-13   & 1.53E-13&9.86E-14
\\
\hline 16  & 2.10E-12  &7.69E-13&2.62E-13
\\
\hline 32 & 2.44E-11  & 7.97E-12 &1.52E-12
\\
\hline
\end{tabular}
\end{center}
\end{table}

\begin{table}[H]
\begin{center}
 \caption{Numerical error and order of convergence for the exact solution $u=1+x+y$ with Dirichlet and Neumann data set on $0\times (0,1)$.}\label{NE:TRI:Case1-1}
 \begin{tabular}{|c|c|c|c|}
\hline $1/h$ &   $\| \nabla  e_0\|_{0} $   & $\|e_0\|_{0}$ & $\3bar e_h \3bar_{h, \Gamma_d}$
\\
\hline 						
          1&  1.69E-15  &8.88E-16& 3.38E-15
\\
\hline 2& 1.10E-14 &6.37E-15&1.10E-14
\\
\hline 4& 3.42E-14 &1.63E-14&1.71E-14
\\
\hline 8& 5.66E-13 &2.56E-13&1.42E-13
\\
\hline 16 &3.41E-12 &1.17E-12&4.26E-13
\\
\hline 32 &2.22E-11 &7.33E-12&1.39E-12
\\
\hline
\end{tabular}
\end{center}
\end{table}

Tables \ref{NE:TRI:Case2-1-2}-\ref{NE:TRI:Case2-3-2} illustrate the performance of the numerical scheme when the boundary conditions are set as follows: (i) both Dirichlet and Neumann boundary conditions on two boundary segments $(0,1)\times 0$ and $1\times (0,1)$, (ii) Dirichlet boundary condition on the boundary segment $0\times (0,1)$, and (iii) Neumann boundary condition on the boundary segment  $(0,1)\times 1$. Tables \ref{NE:TRI:Case2-1-2}-\ref{NE:TRI:Case2-3-2} demonstrate the numerical results for the exact solutions given by $u=\cos(x)\cos(y)$, $u=30xy(1-x)(1-y)$ and $u=\sin(\pi x)\cos(\pi y)$, respectively. The convergence rate in the usual $L^2$ norm for the approximation of $\nabla u_0$ arrives at the order of $\O(h)$.  For the approximation of $u_0$, the convergence rate in the usual $L^2$ norm arrives at the order of $\O(h^2)$.  The convergence rate for $e_h$ in the residual norm arrives at the order of $\O(h)$. The numerical results are in great consistency with the theory established in the previous sections.

\begin{table}[H]
\begin{center}
 \caption{Numerical error and order of convergence for the exact solution $u=\cos(x)\cos(y)$ with Dirichlet and Neumann on $(0,1)\times 0$ and $1\times (0,1)$, Dirichlet on $0\times (0,1)$, and Neumann on $(0,1)\times 1$.}\label{NE:TRI:Case2-1-2}
 \begin{tabular}{|c|c|c|c|c|c|c| }
\hline $1/h$ &  $\|\nabla e_0\|_{0} $  & order & $\| e_0\|_{0} $ & order  & $\3bar e_h \3bar_{h, \Gamma_d}$ & order
\\
\hline 						
1   &0.07776 && 0.1114 & & 1.989 &
\\
\hline
2&	0.03747 &1.053 & 0.02574 &2.113 & 0.7639&1.380
\\
\hline
4& 0.01941 &	0.9487&0.006305 &2.030 & 0.3685 &1.051
\\
\hline
8 &	0.009981 &0.9598&0.001566&2.010 & 0.1823 	&1.015
\\
\hline
16  &	0.005069	&0.9775 &3.90E-04&2.004 & 0.09097 &1.003
\\
\hline
32 &0.002555 &0.9884 &9.75E-05&2.002 & 0.04546 &1.001
\\
\hline
\end{tabular}
\end{center}
\end{table}

\begin{table}[H]
\begin{center}
 \caption{Numerical error and order of convergence for the exact solution $u=30xy(1-x)(1-y)$ with Dirichlet and Neumann on $(0,1)\times 0$ and $1\times (0,1)$, Dirichlet on $0\times (0,1)$, and Neumann on $(0,1)\times 1$.}\label{NE:TRI:Case2-2-1}
 \begin{tabular}{|c|c|c|c|c|c|c| }
\hline $1/h$ &   $\| \nabla  e_0\|_{0} $ & order & $\| e_0\|_{0} $ & order   & $\3bar e_h \3bar_{h, \Gamma_d}$ & order
\\
\hline 						
1&1.343 &&	0.8688 && 29.53  &
\\
\hline
2&0.9050&0.5740 &0.2569 &1.758  &11.32 &1.383
\\
\hline
4&0.4523&1.001&0.07225 &1.830 &5.344 &1.083
\\
\hline
8&0.2251 &1.006 &0.01915 &1.916 &2.634 &1.020
\\
\hline
16&0.1113	&1.017 &0.004913 &1.963 &1.313 &1.005
\\
\hline
32& 0.05504 &1.016 &0.001243 & 1.983 & 0.6557 &1.001
\\
\hline
\end{tabular}
\end{center}
\end{table}

\begin{table}[H]
\begin{center}
 \caption{Numerical error and order of convergence for the exact solution $u=\sin(\pi x)\cos(\pi y)$ with Dirichlet and Neumann on $(0,1)\times 0$ and $1\times (0,1)$, Dirichlet on $0\times (0,1)$, and Neumann on $(0,1)\times 1$.}\label{NE:TRI:Case2-3-2}
 \begin{tabular}{|c|c|c|c|c|c|c| }
\hline $1/h$ &  $\| \nabla  e_0\|_{0}$  & order & $\| e_0\|_{0} $ & order & $\3bar e_h \3bar_{h, \Gamma_d}$ & order
\\
\hline 						
1&1&&1.573&& 9.263  &
\\
\hline
2&  0.4908 &1.680 	&0.1849 &	2.271  &5.086 &0.8651
\\
\hline
4  &0.2187 &1.166  &0.04623 &1.999  & 2.489 &1.031
\\
\hline
8 &0.1052 &1.056 &0.01151 &2.006 &1.237 &1.009
\\
\hline
16 &0.05193 &1.018  &0.002872 &2.003 &0.6172  &1.002
\\
\hline
32 &0.02587 &1.005 &7.17E-04&2.001 &0.3085 &1.001
\\
\hline
\end{tabular}
\end{center}
\end{table}

Tables \ref{NE:TRI:Case3-1-2}-\ref{NE:TRI:Case3-3-1} demonstrate the performance of the PD-WG algorithm when the boundary conditions are set as follows: (i) Dirichlet on two boundary segments $(0,1) \times 0$ and $0\times (0,1)$, and (ii) Neumann on the other two boundary segments $1\times (0,1)$ and $(0,1)\times 1$. Note that this is a standard mixed boundary value problem with no Cauchy data given on the boundary. The purpose of the tests is to show the efficiency of the PD-WG algorithm (\ref{32})-(\ref{2}) for the classical well-posed problems of elliptic type. Tables \ref{NE:TRI:Case3-1-2} - \ref{NE:TRI:Case3-3-1} show the numerical results for the exact solutions given by $u=\cos(x)\cos(y)$,  $u=\sin(x)\sin(y)$, and $u=30xy(1-x)(1-y)$, respectively. The convergence rate in the usual $L^2$ norm for the approximation of $\nabla u_0$ arrives at the order of $\O(h)$. The convergence rate for the approximation of $u_0$ in the usual $L^2$ norm arrives at the order of $\O(h^2)$. The convergence rate for $e_h$ in the scaled residual norm arrives at the order of $\O(h)$. All the numerical results are in consistency with the theory established in the paper.

\begin{table}[H]
\begin{center}
\caption{Numerical error and order of convergence for the exact solution $u=\cos(x)\cos(y)$ with Dirichlet data on $(0,1) \times 0$ and $0\times (0,1)$, and Neumann data on $1\times (0,1)$ and $(0,1)\times 1$.}\label{NE:TRI:Case3-1-2}
 \begin{tabular}{|c|c|c|c|c|c|c| }
\hline $1/h$ & $\| \nabla  e_0\|_{0} $  & order & $\| e_0\|_{0} $ & order  & $\3bar e_h \3bar_{h, \Gamma_d}$ & order
\\
\hline 						
1& 0.05119 && 0.1929 && 1.664 &
\\
\hline
2&0.03565 &0.5218 & 0.04513 &2.096 &0.7543 &1.141
\\
\hline
4 &0.01949 &0.8709 &0.01097 & 2.040 & 0.3671 &1.039
\\
\hline
8   &0.01006 &0.9541 &0.002712 &2.016 &0.1823 &1.010
\\
\hline
16 &0.005099 &0.9806 &6.75E-04&2.007 & 0.0910  &1.003
\\
\hline
32 &0.002564 &0.9917 &1.68E-04&2.004 & 0.04546 & 1.001
\\
\hline
\end{tabular}
\end{center}
\end{table}

 \begin{table}[H]
\begin{center}
 \caption{Numerical error and order of convergence for the exact solution $u=\sin(x)\sin(y)$ with Dirichlet data on $(0,1) \times 0$ and $0\times (0,1)$, and Neumann data on $1\times (0,1)$ and $(0,1)\times 1$.}\label{NE:TRI:Case3-2-1}
 \begin{tabular}{|c|c|c|c|c|c|c| }
\hline $1/h$ &  $\| \nabla e_0\|_{0} $  & order & $\|e_0\|_{0} $ & order  & $\3bar e_h \3bar_{h, \Gamma_d}$ & order
\\
\hline 						
1&0.06759 &&0.05271 & &0.7650 &
\\
\hline
2&0.03395 &0.9932 &0.01664 &1.663 & 0.3056&1.324
\\
\hline
4&0.01814 &0.9044 &0.004519 &1.881 &0.1408 &1.118
\\
\hline
8&0.009424 &0.9448&0.001156 &1.967 & 0.0687 &1.034
\\
\hline
16&0.004785 &0.9778  &2.91E-04&1.991 & 0.03416 &1.009
\\
\hline
32&0.002407&0.9915&7.28E-05&1.997  & 0.01705 &1.002
\\
\hline
\end{tabular}
\end{center}
\end{table}

 \begin{table}[H]
\begin{center}
 \caption{Numerical error and order of convergence for the exact solution $u=30xy(1-x)(1-y)$ with Dirichlet data on $(0,1) \times 0$ and $0\times (0,1)$, and Neumann data on $1\times (0,1)$ and $(0,1)\times 1$.}\label{NE:TRI:Case3-3-1}
 \begin{tabular}{|c|c|c|c|c|c|c| }
\hline $1/h$ &  $\| \nabla e_0\|_{0} $  & order & $\|e_0\|_{0} $ & order    & $\3bar e_h \3bar_{h, \Gamma_d}$ & order
\\
\hline 						
1&1.1234 &&0.8635 & & 29.39 &
\\
\hline
2&0.6085 &0.8845 &0.2094 &2.044 & 11.27&1.383
\\
\hline
4&0.3133 &0.9578 &0.06780&1.627 & 5.337 &1.078
\\
\hline
8&0.1604 &0.9654 &0.01908 &1.829 & 2.634&1.019
\\
\hline
16&0.08045  &0.9959&0.004955 &1.945  & 1.312 &1.005
\\
\hline 32&0.04015&1.003&0.001251 &1.986 & 0.6557&1.001
\\
\hline
\end{tabular}
\end{center}
\end{table}

Tables \ref{NE:TRI:Case4-1-2}-\ref{NE:TRI:Case4-2-2} demonstrate the performance of the PD-WG algorithm (\ref{32})-(\ref{2}) for the elliptic Cauchy problem where the Cauchy boundary conditions are given at two horizontal boundary segments $(0,1)\times 0$ and $(0,1)\times 1$. Tables \ref{NE:TRI:Case4-1-2}-\ref{NE:TRI:Case4-2-2} show the numerical results when the exact solutions are given by $u=\cos(x)\cos(y)$ and $u=30xy(1-x)(1-y)$, respectively. The convergence rate for the approximation of $u_0$ in the usual $L^2$ norm seems to arrive at the order of $\O(h^{1.9})$ which is a little bit lower than the optimal order $\O(h^{2})$. For the exact solution $u=30xy(1-x)(1-y)$, the convergence rate in the usual $L^2$ norm for the approximation of $\nabla u_0$ seems to arrive at the order of $\O(h^{1.2})$ which is better than the expected order of $\O(h)$. The convergence rates for the rest of the numerical results are consistent with what the theory has predicted. Readers are invited to draw their own conclusions for the numerical performance of this set of the numerical results.

 \begin{table}[H]
\begin{center}
 \caption{Numerical error and order of convergence for the exact solution $u=\cos(x)\cos(y)$ with Dirichlet and Neumann data on $(0,1)\times 0$ and $(0,1)\times 1$.}\label{NE:TRI:Case4-1-2}
 \begin{tabular}{|c|c|c|c|c|c|c| }
\hline $1/h$  & $\| \nabla  e_0\|_{0}$ & order   & $\| e_0\|_{0} $ & order & $\3bar e_h \3bar_{h, \Gamma_d}$ & order
\\
\hline 						
1& 0.08163&& 0.1600 & & 1.669 &
\\
\hline
2    &0.04400 &0.8917 &0.04077 &1.972  & 0.7546 &1.145
\\
\hline
4     &0.02211 &0.9925  &0.01014&2.007 & 0.3671 &1.039
\\
\hline
8    &0.01105 &1.000  &0.002615 &1.955 & 0.1823 &1.010
\\
\hline
16&0.005425&1.027 &6.79E-04&1.946  & 0.09097 &1.003
\\
\hline
32   &0.002664&1.026   &1.81E-04&1.906  & 0.04546 &1.001
\\
\hline
\end{tabular}
\end{center}
\end{table}

 \begin{table}[H]
\begin{center}
 \caption{Numerical error and order of convergence for the exact solution $u=30xy(1-x)(1-y)$ with Dirichlet and Neumann data on $(0,1)\times 0$ and $(0,1)\times 1$.}\label{NE:TRI:Case4-2-2}
 \begin{tabular}{|c|c|c|c|c|c|c|c|c|}
\hline $1/h$ &   $\| \nabla  e_0\|_{0} $ & order   &   $\| e_0\|_{0} $ & order  & $\3bar e_h \3bar_{h, \Gamma_d}$ & order
\\
\hline 						
 1 &1.073&& 0.8591&& 29.41  &
 \\
\hline
 2   &0.9207 &0.2208& 0.2340 &1.876  & 11.28&1.383
 \\
\hline
4& 0.7489 &0.2979 & 0.2518&-0.1053 & 5.346 &1.077
 \\
\hline
8 &0.2628&1.511 &0.06500 &1.954& 2.634 &1.021
 \\
\hline
16&  0.1049&1.325& 0.01836&1.824& 1.312 &1.005
 \\
\hline
32&0.04540 &1.208 & 0.004975 &1.884 & 0.6557 &1.001
\\
\hline
\end{tabular}
\end{center}
\end{table}

Tables \ref{NE:TRI:Case5-1-2}-\ref{NE:TRI:Case5-2-2} demonstrate the performance of the PD-WG algorithm (\ref{32})-(\ref{2}) where both the Dirichlet and Neumann boundary conditions are set on two vertical boundary segments $0\times (0,1)$ and $1\times (0,1)$ for the exact solutions $u=xy$ and $u=\cos(x)\sin(y)$, respectively. For the exact solution $u=xy$, the convergence rate of $e_0$ in the usual $L^2$ norm seems to arrive at the order of $\O(h^{2.5})$ and the convergence rate of $e_h$ in the residual norm seems to arrive at the order of $\O(h^2)$, which are much better than the optimal order of $\O(h^{2})$ and $\O(h)$. For the exact solution $u=\cos(x)\sin(y)$, the convergence rate of $\nabla e_0$ in the $L^2$ norm seems to arrive at the order of $\O(h^{1.2})$ which is a little bit higher than the expected order of $\O(h)$; and the convergence rate of $e_0$ in the $L^2$ norm seems to arrive at the order of $\O(h^{1.9})$ which is a little bit lower than the optimal order of $\O(h^2)$. The convergence rates for the rest of the numerical results are in good consistency with the established theory. The interested readers are invited to draw their conclusions for the numerical performance.

 \begin{table}[H]
\begin{center}
 \caption{Numerical error and order of convergence for the exact solution $u=xy$ with Dirichlet and Neumann data on $0\times (0,1)$ and $1\times (0,1)$.}\label{NE:TRI:Case5-1-2}
 \begin{tabular}{|c|c|c|c|c|c|c|}
\hline $1/h$ &  $\| \nabla  e_0\|_{0} $   & order & $\| e_0\|_{0} $ & order  & $\3bar e_h \3bar_{h, \Gamma_d}$ & order
\\
\hline 						
1  &0.07206 &&0.05211&& 0.5185 &
\\
\hline
2 &0.04750 &0.6013 &0.007597 &2.778  & 0.1462 &1.826
\\
\hline
4   &0.02360 &1.009 &0.001270 &2.581 & 0.03830 &1.933
\\
\hline
8   &0.01116 &1.080 &2.34E-04&2.441& 0.009749 &1.974
\\
\hline
16&0.005356 &1.060   &3.94E-05&2.570 & 0.002457 	&1.988
\\
\hline
32  &0.002612&1.036 &6.74E-06&2.546 & 6.17E-04&1.994
\\
\hline
\end{tabular}
\end{center}
\end{table}

\begin{table}[H]
\begin{center}
 \caption{Numerical error and order of convergence for the exact solution $u=\cos(x)\sin(y)$ with Dirichlet and Neumann data on $0\times (0,1)$ and $1\times (0,1)$.}\label{NE:TRI:Case5-2-2}
 \begin{tabular}{|c|c|c|c|c|c|c|c|c|  }
\hline $1/h$ & $\| \nabla  e_0\|_{0} $ & order   & $\|e_0\|_{0} $ & order   & $\3bar e_h \3bar_{h, \Gamma_d}$ & order
\\
\hline 						
1& 0.1225 & &0.1062 & & 0.8855 &
\\
\hline
2 &0.03975 &1.624 &0.02636 &2.011& 0.4474 	& 0.9848
\\
\hline
4 &0.01623 &1.292 &0.006923 &1.929 & 0.2231 & 1.004
\\
\hline
8 &0.006297 &1.366 &0.001745 &1.988 & 0.1114 &1.002
\\
\hline
16&0.002683 &1.231 &4.66E-04&1.907 & 0.05567 &1.001
\\
\hline
32 &0.001210 &1.149&1.25E-04&1.898  & 0.02783  & 1.000
\\
\hline
\end{tabular}
\end{center}
\end{table}

Table \ref{NE:TRI:Case6-1-1} demonstrates the performance of the PD-WG algorithm (\ref{32})-(\ref{2}) for the exact solution $u=\cos(x)\cos(y)$ when the boundary conditions are set as follows: (i) both Dirichlet and Neumann data on the boundary segment $(0,1)\times 0$; (ii) Dirichlet only on the boundary segment $(0,1)\times 1$. The convergence rates in the usual $L^2$ norm for $\nabla e_0$ and $e_0$ arrive at the order of $\O(h)$ and $\O(h^2)$, respectively. The convergence rate for $e_h$ in the scaled residual norm arrives at the order of $\O(h)$. The numerical results are in perfect consistency with the theory established in the previous sections.

\begin{table}[H]
\begin{center}
 \caption{Numerical error and order of convergence for the exact solution $u=\cos(x)\cos(y)$ with Dirichlet and Neumann data on $(0,1)\times 0$, and Dirichlet only on $(0,1)\times 1$.}\label{NE:TRI:Case6-1-1}
 \begin{tabular}{|c|c|c|c|c|c|c| }
\hline $1/h$  & $\| \nabla  e_0\|_{0} $ & order  & $\|e_0\|_{0} $ & order  & $\3bar e_h \3bar_{h, \Gamma_d}$ & order
\\
\hline 						
1&0.07672 &&0.2065 && 1.667 &
\\
\hline
2&0.04154 &0.8850 &0.04355 &2.246 & 0.7542 &1.145
\\
\hline
4&0.02206&0.9132 &0.01240 &1.812 & 0.3671 &1.039
\\
\hline
8&0.0105 &1.073&0.003102&1.999 & 0.1823&1.010
\\
\hline
16&0.005245 &0.9997 &8.27E-04&1.906 & 0.0910&1.003
\\
\hline
\end{tabular}
\end{center}
\end{table}

Table \ref{NE:TRI:Case8-1-1} demonstrates the performance of the PD-WG algorithm (\ref{32})-(\ref{2}) when the Dirichlet and Neumann
boundary conditions are set on the boundary segment $(0,1)\times 0$. These numerical results illustrate that the convergence rate for the solution of the primal-dual weak Galerkin algorithm in the residual norm is at the rate of $\O(h)$, which is in great consistency with the theory established in this paper.

\begin{table}[H]
\begin{center}
\caption{Numerical error and order of convergence for the exact
solutions $u_1=\sin(x)\sin(y)$, $u_2=\cos(x)\cos(y)$ and
$u_3=\cos(x)\sin(y)$ with Dirichlet and Neumann data on $(0,1)\times 0$.}\label{NE:TRI:Case8-1-1}
 \begin{tabular}{|c|c|c|c|c|c|c|}
\hline $\frac{1}{h}$ & $\3bar e_h\3bar_{h, \Gamma_d}$ for $u_1$  & order&   $\3bar e_h\3bar_{h, \Gamma_d} $ for $u_2$  & order  & $\3bar e_h\3bar_{h, \Gamma_d}$ for $u_3$  & order
\\
 \hline
 1&1.006  && 1.657 &&0.9056	&
 \\
 \hline
 2&0.3357 &1.584 	&0.7572	&1.130 &0.4490 &1.012
 \\
 \hline
 4&0.1599 &1.070 &0.3678&1.042&0.2241 &1.003
\\
 \hline
 8&0.07354 &1.121 &0.1825 &1.011 &0.1116 &1.006
 \\
 \hline
 16&0.03563 &1.045 &0.09102 &1.003 &0.05573 &1.001
 \\
 \hline
 32&0.01751 &1.025 &0.04548 &1.001&0.02785 &1.001
 \\
 \hline
 \end{tabular}
\end{center}
\end{table}

\section{Concluding Remarks}
In conclusion, the numerical approximations arising from the primal-dual weak Galerkin finite element scheme (\ref{32})-(\ref{2}) are convergent to the exact solution at rates that are consistent with the theoretical predictions in the scaled residual norm.
In some of the numerical test cases, their convergence rates in the usual $L^2$ and $H^1$ norms seem to be slightly lower than the optimal order. Note that the estimate in Theorem \ref{Thm:L2errorestimate} was established in a weak $L^2$ topology which may not provide the usual $L^2$ error estimate due to the solvability of the auxiliary problem \eqref{dual1} for arbitrary input function $\eta$. We suspect that the loss on the rate of convergence in $L^2$ and $H^1$ norms might be caused by the ill-posedness of the elliptic Cauchy problem or the poor conditioning of the discrete linear system resulted from the scheme (\ref{32})-(\ref{2}). Nevertheless, the PD-WG finite element method (\ref{32})-(\ref{2}) does provides one and only one numerical solution even if the original elliptic Cauchy problem is not well-posed or does not have any solutions. This numerical approximation is theoretically convergent to the exact solution in a mesh-dependent scaled residual norm. Overall, we are confident that the primal-dual weak Galerkin finite element method is a reliable and robust numerical method for the ill-posed elliptic Cauchy problem.


\begin{thebibliography}{99}
\bibitem{babuska} {\sc I. Babu\u{s}ka}, {\em The finite element method with Lagrange multipliers}, Numer. Math., vol. 20, pp. 179-192, 1973.

\bibitem{b1974} {\sc F. Brezzi}, {\em On the existence, uniqueness, and approximation of saddle point problems arising from Lagrange multipliers}, RAIRO, 8 (1974), pp. 129-151.

 \bibitem{ErikBurman-EllipticCauchy} {\sc E. Burman}, {\em Error estimates for stabilized finite element methods applied to ill-posed problems}, C. R. Acad. Sci. Paris, Ser., vol. I 352,
pp. 655-659, 2014. http://dx.doi.org/10.1016/j.crma.2014.06.008

\bibitem{Burman} {\sc E. Burman}, {\em Stabilized finite element methods for nonsymmetric, noncoercive, and ill-posed problems. Part I: Elliptic equations}, SIAM J. Sci. Comput., vol. 35, pp. 2752-2780, 2013.

\bibitem{ErikBurman-EllipticCauchy-SIAM02} {\sc E. Burman}, {\em Stabilized finite element methods for nonsymmetric, noncoercive, and ill-possed problems. Part II: hyperbolic equations}, SIAM J. Sci. Comput, vol. 36, No. 4, pp. A1911-A1936, 2014.

\bibitem{Gilbarg-Trudinger} {\sc D. Gilbarg and N. S. Trudinger}. {\em Elliptic Partial Differential Equations of Second Order}. Springer-Verlag, Berlin, second edition, 1983.

\bibitem{mwy3655} {\sc L. Mu, J. Wang, and X. Ye}, {\em Weak Galerkin finite element methods on polytopal meshes}, International Journal of Numerical Analysis and Modeling, vol. 12, pp. 31-53, 2015.

\bibitem{ww2016} {\sc C. Wang and J. Wang}, {\em A primal-dual weak Galerkin finite element method for second order elliptic equations in non-divergence form},  Mathematics of Computation, Math. Comp., vol. 87, pp. 515-545, 2018.

\bibitem{ww2017} {\sc  C. Wang and J. Wang}, {\em A Primal-Dual weak Galerkin finite element method for Fokker-Planck type equations}, arXiv:1704.05606, SIAM Journal of Numerical Analysis, accepted.

\bibitem{ww2018} {\sc C. Wang and J. Wang}, {\em Primal-Dual Weak Galerkin Finite Element Methods for Elliptic Cauchy Problems},  submitted. arXiv:1806.01583.

\bibitem{WangYe_2013}
{\sc J. Wang and X. Ye},
\newblock{\em A weak Galerkin mixed finite element method for second-order ellliptic problems}, arXiv: 1104.2897vl, {J. Comp. and Appl. Math.}, {241}, 103-115, 2013.

\bibitem{wy3655} {\sc J. Wang and X. Ye}, {\em A weak Galerkin mixed finite element method for second-order elliptic problems}. Math. Comp., vol. 83, pp. 2101-2126, 2014.

\end{thebibliography}
\end{document}